\documentclass[12pt]{article}

\usepackage{latexsym}
\usepackage{a4wide}
\usepackage{amscd}
\usepackage{graphics}
\usepackage{amsmath}
\usepackage{amssymb}
\usepackage{mathrsfs}
\usepackage{amsthm}
\usepackage{bbm}
\usepackage{color}
\usepackage{accents}
\usepackage{enumerate}
\usepackage{fancyhdr}
\usepackage{datetime}
\input xy
\xyoption{all}

\usepackage[margin=1in]{geometry}% http://ctan.org/pkg/geometry

\newcommand{\N}{\mathbb{N}}                     % the natural numbers
\newcommand{\Z}{\mathbb{Z}}                     % the integer numbers
\newcommand{\R}{\mathbb{R}}                     % the real line
\newcommand{\C}{\mathbb{C}}                     % the complex plane
\newcommand{\T}{\mathbb{T}}                     % the torus
\renewcommand{\L}{\mathcal{L}}
\newcommand{\h}{\mathcal{H}}
\newcommand{\V}{\mathcal{V}}
\newtheorem{thm}{\sc Theorem}[section]               % numbered absolutely
\newtheorem*{thm*}{\sc Theorem}               % no number
\newtheorem{cor}[thm]{\sc Corollary}        % numbered along with Theorem
\newtheorem*{cor*}{\sc Corollary}        % no number
\newtheorem{lem}[thm]{\sc Lemma}            % numbered along with Theorem
\newtheorem{prop}[thm]{\sc Proposition}     % numbered along with Theorem
\newtheorem{defn}[thm]{\sc Definition}      % numbered along with Theorem
\newtheorem{rem}[thm]{\sc Remark}           % numbered along with Theorem
\title{A systolic inequality for geodesic flows\\on the two-sphere}

\author{Alberto Abbondandolo, Barney Bramham,\\Umberto L. Hryniewicz and Pedro A. S. Salom\~ao}

\AtEndDocument{\bigskip{\footnotesize%
  (Alberto Abbondandolo) \textsc{Ruhr Universit\"at Bochum, Fakult\"at f\"ur Mathematik, Geb\"aude NA 4$\slash$33, D--44801 Bochum, Germany} \par  
  \textit{E-mail address}: \texttt{alberto.abbondandolo@rub.de} \par
  \addvspace{\medskipamount}
  (Barney Bramham) \textsc{Ruhr Universit\"at Bochum, Fakult\"at f\"ur Mathematik, Geb\"aude NA 5$\slash$32, D--44801 Bochum, Germany} \par  
  \textit{E-mail address}: \texttt{barney.bramham@rub.de} \par
  \addvspace{\medskipamount}
  (Umberto L.~Hryniewicz) \textsc{Universidade, Federal do Rio de Janeiro -- Departamento de Matem\'atica Aplicada, Av.\ Athos da 
  Silveira Ramos 149, Rio de Janeiro RJ, Brazil 21941-909} \par  
  \textit{E-mail address}: \texttt{umberto@labma.ufrj.br} \par
  \addvspace{\medskipamount}
   (Pedro A.~S.~Salom\~ao) \textsc{Universidade de S\~ao Paulo, Instituto de Matem\'atica e Estat\'istica -- Departamento de Matem\'atica, 
   Rua do Mat\~ao, 1010 - Cidada Universit\~aria - S\~ao Paulo SP, Brazil 05508-090} \par  
  \textit{E-mail address}: \texttt{psalomao@ime.usp.br} 
}}

\date{}

\begin{document}

\maketitle

\abstract{For a Riemannian metric $g$ on the two-sphere, let $\ell_{\min}(g)$ be the length of the shortest closed geodesic and $\ell_{\max}(g)$ be the length of the longest simple closed geodesic.
We prove that if the curvature of $g$ is positive and sufficiently pinched, then the sharp systolic inequalities
\[
\ell_{\rm min}(g)^2 \leq \pi \ {\rm Area}(S^2,g) \leq \ell_{\max}(g)^2,
\]
hold, and each of these two inequalities is an equality if and only if the metric $g$ is Zoll. The first inequality answers positively a conjecture of Babenko and Balacheff. The proof combines arguments from Riemannian and symplectic geometry.}

\tableofcontents

\section*{Introduction}
\addcontentsline{toc}{section}{\numberline{}Introduction}

In 1988 Croke proved that the length of the shortest closed geodesic on a Riemannian two-sphere can be bounded from above in terms of its area: there exists a positive number $C$ such that the quantity
\[
\ell_{\min}(g) := \text{length of the shortest non-constant closed geodesic on }  (S^2,g)
\]
is bounded from above by
\[
\ell_{\min}(g)^2 \leq C \ {\rm Area}(S^2,g),
\]
for every Riemannian metric $g$ (see \cite{cro88}). In other words, the {\em systolic ratio}
\[
\rho_{\mathrm{sys}}(g) := \frac{\ell_{\min}(g)^2}{\mathrm{Area}(S^2,g)}
\]
is bounded from above on the space of all Riemannian metrics on $S^2$. The value of the supremum of $\rho_{\mathrm{sys}}$ is not known, but it was shown to be not larger than $32$ by Rotman~\cite{rot06}, who improved the previous estimates due to Croke \cite{cro88}, Nabutowski and Rotman \cite{nr02}, and Sabourau \cite{sab04}.

The na\"ive conjecture that the round metric $g_{\mathrm{round}}$ on $S^2$ maximises $\rho_{\mathrm{sys}}$ is false. Indeed, 
\[
\rho_{\mathrm{sys}}(g_{\mathrm{round}}) = \pi,
\]
while, by studying suitable metrics approximating a singular metric constructed by gluing two flat equilateral triangles along their boundaries, one sees that
\[
\sup \rho_{\mathrm{sys}} \geq 2\sqrt{3} > \pi.
\]
This singular example is known as the Calabi-Croke sphere.
Actually, it is conjectured that the supremum of $\rho_{\mathrm{sys}}$ is $2\sqrt{3}$ and that it is not attained. See \cite{bal10} and \cite{sab10} for two different proofs of the fact that the Calabi-Croke sphere can be seen as a local maximiser of $\rho_{\mathrm{sys}}$.

In this paper, we are interested in the behaviour of $\rho_{\mathrm{sys}}$ near the round metric $g_{\mathrm{round}}$ on $S^2$. To the authors' knowledge, this question was first raised by Babenko, and then studied by Balacheff, who in \cite{bal06} showed that $g_{\mathrm{round}}$ can be seen as a critical point of $\rho_{\mathrm{sys}}$. Balacheff also conjectured the round metric to be a local maximiser of $\rho_{\mathrm{sys}}$ and gave some evidence in favour of this conjecture (see also \cite[Question 8.7.2]{bm13}). Certainly, $g_{\mathrm{round}}$ is not a strict local maximiser of $\rho_{\mathrm{sys}}$, even after modding out rescaling: in any neighbourhood of it there are infinitely many non-isometric Zoll metrics, i.e.\ Riemannian metrics on $S^2$ all of whose geodesics are closed and have the same length, and $\rho_{\mathrm{sys}}$ is constantly equal to $\pi$ on them (see \cite{wei74}, \cite{gui76} and Appendix \ref{app_conj_zoll} below). Further evidence in favour of the local maximality of the round metric is given in \cite{apb14}, where \'Alvarez Paiva and Balacheff prove that $\rho_{\mathrm{sys}}$ strictly decreases under infinitesimal deformations of the round metric which are not tangent with infinite order to the space of Zoll metrics.

The aim of this paper is to give a positive answer to Babenko's and Balacheff's conjecture and to complement it with a statement about the length $\ell_{\max}(g)$ of the longest {\em simple} closed geodesic on $(S^2,g)$. The latter number is well defined whenever the Gaussian curvature $K$ of $(S^2,g)$ is non-negative, see \cite{cc92}. 

We 
recall that a Riemannian metric $g$ on $S^2$ is $\delta$-pinched, for some $\delta\in (0,1]$, if its Gaussian curvature $K$ is positive and satisfies
\[
\min K \geq \delta \max K.
\]
The main result of this article is the following:
 
\begin{thm*}
Let $g$ be a $\delta$-pinched smooth Riemannian metric on $S^2$, with
\[
\delta > \frac{4+\sqrt 7}{8} = 0.8307\dots
\]
Then
\[
\ell_{\min}(g)^2 \leq \pi \, {\mathrm{Area}}(S^2,g) \leq  \ell_{\max}(g)^2.
\]
Each of the two inequalities is an equality if and only if $g$ is Zoll.
\end{thm*}

Therefore, when the curvature of $g$ satisfies the above pinching condition, then
\[
\rho_{\mathrm{sys}}(g) \leq \rho_{\mathrm{sys}}(g_{\mathrm{round}}) = \pi,
\]
with the equality holding if and only if $g$ is Zoll. In particular, the round metric is a local maximiser of the systolic ratio in the $C^2$ topology of metrics.

As far as we know, also the lower bound for the length $\ell_{\max}(g)$ of the longest simple closed geodesic which is stated in the above theorem is new. Lower bounds for $\ell_{\max}(g)$ are studied by Calabi and Cao in the already mentioned \cite{cc92}, where the non-sharp bound
\[
\ell_{\max}(g)^2 \geq \frac{\pi}{2} \, {\mathrm{Area}}(S^2,g)
\]
is proved for any metric $g$ with non-negative curvature. This bound is deduced by the following sharp lower bound in terms of the diameter
\[
\sup \{ \ell(\gamma) \mid \gamma \mbox{ simple closed geodesic on } (S^2,g) \} \geq 2 \, \mathrm{diam} (S^2,g),
\]
which is due to Croke and holds for any metric (when finite, this supremum is a maximum; the supremum is finite in the case $K\geq 0$). Unlike for the first inequality, we do not have counterexamples to the second inequality in our main theorem for metrics which are far from the round one. Our theorem also implies that, under the pinching assumption, when all the simple closed geodesics have the same length the metric must be Zoll.

The proof of the above theorem combines arguments from Riemannian geometry and techniques from  symplectic geometry. The role of symplectic geometry in the proof should not surprise: as stressed in \cite{apb14}, the systolic ratio $\rho_{\mathrm{sys}}$ is a symplectic invariant, meaning that if two metrics give rise to geodesic flows on the cotangent bundle of $S^2$ which are conjugate by a symplectic diffeomorphism, then their systolic ratios coincide. Another argument in favour of the symplectic nature of our theorem is that Zoll metrics, which produce the extremal cases of both our inequalities, are in general not pairwise isometric, but their geodesic flows are symplectically conjugate, see Appendix~\ref{app_conj_zoll}. The presence of a large set of not pairwise isometric local maximisers for $\rho_{\mathrm{sys}}$ seems to exclude the possibility of a purely Riemannian geometric proof.

\medskip

We conclude this introduction with an informal description of the proof. We start by looking at a closed geodesic $\gamma$ on $(S^2,g)$ of minimal length $L=\ell_{\min}(g)$, parametrised by arc length. When the curvature  of $(S^2,g)$ is non-negative, this curve is simple (see \cite{cc92}, or Lemma \ref{simple} below for a proof  under the assumption that $g$ is $\delta$-pinched for some $\delta>1/4$).

Then we consider a Birkhoff annulus $\Sigma_{\gamma}^+$ which is associated to $\gamma$: $\Sigma_{\gamma}^+$ is the set of all unit tangent vectors to $S^2$ which are based at points of $\gamma(\R)$ and point in the direction of one of the two disks which compose $S^2 \setminus \gamma(\R)$. The set $\Sigma_{\gamma}^+$ is a closed annulus, and its boundary consists of the unit vectors $\dot\gamma(t)$ and $-\dot\gamma(t)$, for $t\in \R/L\Z$.

By a famous result of Birkhoff, the positivity of the curvature $K$ guarantees that the orbit of any $v$ in the interior part of $\Sigma_{\gamma}^+$ under the geodesic flow on the unit tangent bundle  $T^1 S^2$ of $(S^2,g)$ hits $\Sigma_{\gamma}^+$ again at some positive time. This allows us to consider the first return time function 
\[
\tau: \mathrm{int}(\Sigma_{\gamma}^+) \rightarrow (0,+\infty), \qquad \tau(v) := \inf \{ t> 0 \; | \; \phi_t(v)\in \Sigma_{\gamma}^+ \},
\]
and the first return time map
\[
\varphi: \mathrm{int}(\Sigma_{\gamma}^+) \rightarrow \mathrm{int}(\Sigma_{\gamma}^+), \qquad \varphi(v) := \phi_{\tau(v)}(v),
\]
where $\phi_t : T^1 S^2 \rightarrow T^1 S^2$ denotes the geodesic flow induced by $g$. The function $\tau$ and the map $\varphi$ are smooth and, as we will show, extend smoothly to the boundary of $\Sigma_{\gamma}^+$. 

The map $\varphi$ preserves the two-form $d\lambda$, where $\lambda$ is the restriction to $\Sigma_{\gamma}^+$ of the standard contact form on $T^1 S^2$. The two-form $d\lambda$ is an area-form in the interior of $\Sigma_{\gamma}^+$, but vanishes on the boundary, due to the fact that the geodesic flow is not transverse to the boundary. Indeed, if we consider the coordinates
\[
(x,y)\in \R/L \Z\times [0,\pi]
\] 
on $\Sigma_{\gamma}^+$ given by the arc parameter $x$ on the geodesic $\gamma$ and the angle $y$ which a unit tangent vector makes with $\dot{\gamma}$, the one-form $\lambda$ and its differential have the form
\begin{equation}
\label{theform}
\lambda = \cos y\, dx, \qquad d\lambda = \sin y\, dx\wedge dy.
\end{equation}

By lifting the first return map $\varphi$ to the strip $S=\R \times [0,\pi]$, we obtain a diffeomorphism $\Phi:S \rightarrow S$ which preserves the two-form $d\lambda$ given by (\ref{theform}), maps each boundary component into itself, and satisfies
\[
\Phi(x+L,y) = (L,0) + \Phi(x,y), \qquad \forall (x,y)\in S.
\]
As we shall see, diffeomorphisms of $S$ with these properties have a well defined {\em flux} and, when the flux vanishes, a well defined {\em Calabi invariant}. The {\em flux} of $\Phi$ is its average horizontal displacement. We shall prove that, if $g$ is $\delta$-pinched with $\delta>1/4$, one can find a lift $\Phi$ of $\varphi$ having zero flux. For diffeomorphisms $\Phi$ with zero flux, the action and the Calabi invariant can be defined in the following way. The {\em action} of $\Phi$ is the unique function
\[
\sigma: S \rightarrow \R,
\]
such that 
\[
d\sigma=\Phi^* \lambda - \lambda \qquad \mbox{on } S,
\]
and whose value at each boundary point $w\in \partial S$ coincides with the integral of $\lambda$ on the arc from $w$ to $\Phi(w)$ along $\partial S$. The {\em Calabi invariant} of $\Phi$ is the average of the action, that is, the number
\[
\mathrm{CAL}(\Phi) = \frac{1}{2L} \iint_{[0,L]\times [0,\pi]} \sigma\, d\lambda.
\]
We shall prove that, still assuming $g$ to be $\delta$-pinched with $\delta>1/4$, the action and the Calabi invariant of $\Phi$ are related to the geometric quantities we are interested in by the identities
\begin{eqnarray}
\label{uno}
\tau\circ p &=& L + \sigma, \\
\label{due}
\pi \, \mathrm{Area} (S^2,g) &=& L^2 + L\ \mathrm{CAL}(\Phi),
\end{eqnarray}
where 
\[
p: S= \R \times [0,\pi] \rightarrow \Sigma_{\gamma}^+ = \R/L\Z \times [0,\pi] 
\]
is the standard projection.
The $\delta$-pinching assumption on $g$ with $\delta> (4+\sqrt{7})/8$ implies that the map $\Phi$ is {\em monotone}, meaning that, writing 
\[
\Phi(x,y) = (X(x,y),Y(x,y)),
\]
the strict inequality $D_2 Y>0$ holds on $S$. This is proved by using an upper bound on the perimeter of convex geodesic polygons which follows from Toponogov's comparison theorem. This upper bound plays an important role also in the proof of some of the other facts stated above, and we discuss it in the appendix which concludes this article. The monotonicity of $\Phi$ allows us to represent it in terms of a generating function. By using such a generating function, we shall prove the following fixed point theorem (Theorem \ref{thm_calabi}): If a monotone map $\Phi$ with vanishing flux is not the identity and satisfies $\mathrm{CAL}(\Phi)\leq 0$, then $\Phi$ has an interior fixed point with negative action. 

The first inequality in our main theorem is now a consequence of the latter fixed point theorem and of the identities (\ref{uno}) and (\ref{due}). First one observes that $\Phi$ is the identity if and only if $g$ is Zoll. Assume that $g$ is not Zoll. If, by contradiction, the inequality
\[
L^2 = \ell_{\min}(g)^2 \geq \pi \, {\mathrm{Area}}(S^2,g)
\]
holds, (\ref{due}) implies that $\mathrm{CAL}(\Phi)\leq 0$, so $\Phi$ has a fixed point $w\in \mathrm{int}(S)$ with $\sigma(w)<0$. But then (\ref{uno}) implies that the closed geodesic which is determined by $p(w)\in \Sigma_{\gamma}^+$ has length $\tau(p(w)) < L$, which is a contradiction, because $L$ is the minimal length of a closed geodesic. This shows that when $g$ is not Zoll, the strict inequality
\[
\ell_{\min}(g)^2 < \pi \, {\mathrm{Area}}(S^2,g)
\]
holds. This proves the first inequality. The proof of the second one uses the Birkhoff map associated to a simple closed geodesic of maximal length and is similar.

\medskip

\noindent{\bf Acknowledgments.} The present work is part of A.A.'s activities within CAST, a Research Network Program of the European 
Science Foundation.  A.A. wishes to thank Juan Carlos \'Alvarez Paiva for sharing with him his view of systolic geometry. U.H.\ is grateful to Samuel Senti for endless interesting discussions about the relations between systolic inequalities and ergodic theory, and for his interest in the paper.   U.H.\ also acknowledges support from CNPq Grant 309983/2012-6.  
P.S.\ is partially supported by CNPq Grant no. 301715/2013-0 and FAPESP Grant no. 2013/20065-0

\section{A class of self-diffeomorphisms of the strip preserving a two-form}
\label{sez1}

We denote by $S$ the closed strip
\[
S := \R \times [0,\pi],
\]
on which we consider coordinates $(x,y)$, $x\in \R$, $y\in [0,\pi]$. The smooth two-form
\[
\omega (x,y) := \sin y \, dx\wedge dy
\]
is an area form on the interior of $S$ and vanishes on its boundary. 
Fix some $L>0$, and let $\mathcal{D}_L(S,\omega)$ be the group of all diffeomorphisms $\Phi: S \rightarrow S$ such that:
\begin{enumerate}[(i)]
\item $\Phi(x+L,y) = (L,0) + \Phi(x,y)$ for every $(x,y)\in S$.
\item $\Phi$ maps each component of $\partial S$ into itself.
\item $\Phi$ preserves the two-form $\omega$.
\end{enumerate}

The elements of $\mathcal{D}_L(S,\omega)$ are precisely the maps which are obtained by lifting to the universal cover
\[
S \rightarrow A:= \R/L\Z \times [0,\pi]
\]
self-diffeomorphisms of $A$ which preserve the two-form $\omega$ on $A$ and map each boundary component into itself. 

By conjugating an element $\Phi$ of $\mathcal{D}_L(S,\omega)$ by the homeomorphism
\[
S \rightarrow \R \times [-1,1], \qquad (x,y) \mapsto (x,-\cos y),
\]
one obtains a self-homeomorphism of the strip $\R \times [-1,1]$ which preserves the standard area form $dx\wedge dy$. Such a homeomorpshism is in general not continuously differentiable up to the boundary. Since we find it more convenient to work in the smooth category, we prefer not to use the above conjugacy and to deal with the non-standard area-form $\omega$ vanishing on the boundary.

\subsection{The flux and the Calabi invariant}

In this section, we define the flux on $\mathcal{D}_L(S,\omega)$ and the Calabi homomorphism on the kernel of the flux. These real valued homomorphisms were introduced by Calabi in \cite{cal70} for the group of compactly supported symplectic diffeomorphisms of symplectic manifolds of arbitrary dimension. See also \cite[Chapter 10]{ms98}. In this paper we need to extend these definitions to the surface with boundary $S$. Our presentation is self-contained.

\begin{defn}
The flux of a map $\Phi\in \mathcal{D}_L(S,\omega)$, $\Phi(x,y)=(X(x,y),Y(x,y))$, is the real number
\[
\mathrm{FLUX}(\Phi) := \frac{1}{2L} \iint_{[0,L]\times [0,\pi]} (X(x,y)-x)\, \omega(x,y).
\]
\end{defn}

In other words, the flux of $\Phi$ is the average shift in the horizontal direction (notice that $2L$ is the total area of $[0,L]\times [0,\pi]$ with respect to the area form $\omega$). 
Using the fact that the elements of $\mathcal{D}_L(S,\omega)$ preserve $\omega$, it is easy to show that the function $\mathrm{FLUX}: \mathcal{D}_L(S,\omega) \rightarrow \R$ is a homomorphism. 

\begin{prop}
\label{flux-prop}
Let $\alpha_0: [0,\pi] \rightarrow S$ be the path $\alpha_0(t):=(0,t)$. Then
\[
\mathrm{FLUX}(\Phi) = \frac{1}{2} \int_{\Phi(\alpha_0)} x\sin y \, dy,
\]
for every $\Phi$ in $\mathcal{D}_L(S,\omega)$.
\end{prop}

\begin{proof}
Let $\Theta :S\rightarrow S$ be the covering transformation $(x,y)\mapsto(x+L,y)$, and 
 set $Q:=[0,L]\times[0,\pi]$.  With its natural orientation, $Q\subset S$ is the region whose signed 
boundary is $\Theta(\alpha_0)-\alpha_0$ plus pieces that lie in $\partial S$.  
Since $\Phi\in\mathcal{D}_L(S,\omega)$ commutes with $\Theta$, we have 
\begin{equation}\label{E:flux-prop-1}
		      \Phi(Q)-Q=\Theta(R)-R
\end{equation}
as simplicial $2$-chains in $S$, where $R \subset S$ is an oriented region whose signed boundary 
consists of $\Phi(\alpha_0)-\alpha_0$ plus two additional pieces in $\partial S$ that we do not need to label.  Therefore, 
\[
	      \mathrm{FLUX}(\Phi) = \frac{1}{2L} \int_{Q} \big(X-x\big)\,\omega 
			= \frac{1}{2L} \int_{Q} \bigl( \Phi^*(x\, \omega)-x\,\omega  \bigr)
				= \frac{1}{2L} \int_{R} \bigl( \Theta^*(x\, \omega)-x\, \omega \bigr), 
\]
using (\ref{E:flux-prop-1}) for the last equality.  Since
\[
\Theta^*(x\, \omega)-x\, \omega = L\, \omega=L\, d\big(x\sin y\,dy\big),
\]
by Stokes theorem we conclude that  
\[
	      \mathrm{FLUX}(\Phi) = \frac{1}{2} \int_{\partial R} x\sin y\,dy 
				= \frac{1}{2} \int_{\Phi(\alpha_0)-\alpha_0} x\sin y\,dy 
				= \frac{1}{2} \int_{\Phi(\alpha_0)} x\sin y\,dy.  
\]
\end{proof}

\begin{rem}
More generally, it is not difficult to show that if $\alpha$ is any smooth path in $S$ with the first end-point in $\R\times \{0\}$ and the second one in $\R \times \{\pi\}$, then
\[
\mathrm{FLUX} (\Phi) = \frac{1}{2} \int_{\Phi(\alpha)} x\sin y \, dy - \frac{1}{2} \int_{\alpha} x\sin y \, dy,
\]
for every $\Phi$ in $\mathcal{D}_L(S,\omega)$.
\end{rem}

Now we fix the following primitive of $\omega$ on $S$
\[
\lambda:= \cos y\, dx.
\]
Notice that $\lambda$ is invariant with respect to translations in the $x$-direction.
Let $\Phi$ be an element of $\mathcal{D}_L(S,\omega)$. Since $\Phi$ preserves $\omega=d\lambda$, the one-form
\[
\Phi^* \lambda - \lambda
\]
is closed. Since $S$ is simply connected, there exists a unique smooth function
\[
\sigma: S \rightarrow \R
\]
such that
\begin{equation}
\label{action1}
d\sigma = \Phi^* \lambda - \lambda \qquad \mbox{on } S,
\end{equation}
and
\begin{equation}
\label{action2}
\sigma(0,0) = \int_{\gamma_0} \lambda - \mathrm{FLUX}(\Phi),
\end{equation}
where $\gamma_0$ is a smooth path in $\partial S$ going from $(0,0)$ to $\Phi(0,0)$. Of course, the value of the integral in (\ref{action2}) does not depend on the choice of $\gamma_0$, but only on its end-points.

Notice that the function $\sigma$ is $L$-periodic in the first variable: This follows from the fact that $\Phi^* \lambda-\lambda$ is $L$-periodic in the first variable and its integral on the path $\beta_0: [0,L] \rightarrow S$, $\beta_0(t)=(t,0)$, vanishes:
\[
\int_{\beta_0} ( \Phi^* \lambda - \lambda ) = \int_{\Phi(\beta_0)} \lambda - \int_{\beta_0} \lambda = \int_{\Phi(0,0) + \beta_0} \lambda -  \int_{\beta_0} \lambda = 0,
\]
thanks to the invariance of $\lambda$ with respect to horizontal translations (here, the $L$-periodicity of $\lambda$ in the first variable would have sufficed).

Notice also that, thanks to (\ref{action1}), the same normalization condition (\ref{action2}) holds for every point in the lower component of the boundary of $S$: For every $x$ in $\R$ there holds
\begin{equation}
\label{action3}
\sigma(x,0) = \int_{\gamma_x} \lambda- \mathrm{FLUX}(\Phi),
\end{equation}
where $\gamma_x$ is a smooth path in $\partial S$ going from $(x,0)$ to $\Phi(x,0)$. Indeed, if $\xi_x$ is a smooth path in $\partial S$ from $(0,0)$ to $(x,0)$, then the paths $\gamma_0 \# (\Phi\circ \xi_x)$ and $\xi_x \# \gamma_x$ in $\partial S$ have the same end-points. Thus,
\[
\int_{\gamma_0} \lambda+ \int_{\xi_x} \Phi^* \lambda = \int_{\xi_x} \lambda + \int_{\gamma_x} \lambda,
\]
and equations (\ref{action1}) and (\ref{action2}) imply
\[
\sigma(x,0) = \sigma(0,0) + \int_{\xi_x} d\sigma = \int_{\gamma_0} \lambda  - \mathrm{FLUX}(\Phi) + \int_{\xi_x} ( \Phi^* \lambda - \lambda) = \int_{\gamma_x} \lambda - \mathrm{FLUX}(\Phi).
\]
Therefore, we can give the following definitions.

\begin{defn}
\label{action-defn}
Let $\Phi\in \mathcal{D}_L(S,\omega)$.
The unique smooth function $\sigma:S \rightarrow \R$ which satisfies (\ref{action1}) and (\ref{action2}) (or, equivalently, (\ref{action1}) and (\ref{action3})) is called action of $\Phi$.
\end{defn}

\begin{defn}
Let $\Phi\in \ker \mathrm{FLUX}$ and let $\sigma$ be the action of $\Phi$. The Calabi invariant of $\Phi$ is the real number
\[
\mathrm{CAL}(\Phi) = \frac{1}{2L} \iint_{[0,L]\times [0,\pi]} \sigma\, \omega.
\]
\end{defn}

In other words, the Calabi invariant of $\Phi$ is its average action.  The following remark explains why we define the Calabi invariant only for diffeomorphisms having zero flux.

\begin{rem}
The action $\sigma$ depends on the choice of the primitive $\lambda$ of $\omega$. Let $\lambda'$ be another primitive of $\omega$, still $L$-periodic in the first variable. Then one can easily show that $\lambda'=\lambda+df + c\, dx$, where $f:S\rightarrow \R$ is a smooth function which is $L$-periodic in the first variable and $c$ is a real number, and that the action $\sigma'$ of $\Phi$ with respect to $\lambda'$ is given by
\[
\sigma'(x,y) = \sigma(x,y) + f\circ \Phi(x,y) - f(x,y) + c (X(x,y)-x),
\]
where $\Phi=(X,Y)$. If $\Phi$ has zero flux, then the integrals of $\sigma' \, \omega$ and of $\sigma\, \omega$ on $[0,L]\times [0,\pi]$ coincide, so the Calabi invariant of $\Phi$ does not depend on the choice of the periodic primitive of $\omega$. Moreover, this formula also shows that the value of the action at a fixed point of $\Phi$ is independent on the choice of the primitive of $\omega$.
Since $\Phi^* \lambda$ is another periodic primitive of $\omega$, the above facts imply that $\mathrm{CAL}:\ker \mathrm{FLUX} \rightarrow \R$ is a homomorphsim.   In this paper, we work always with the chosen primitive $\lambda$ of $\omega$ and do not need the homomorphsim property of $\mathrm{CAL}$, so we leave these verifications to the reader. See \cite{fat80} and \cite{gg95} for interesting equivalent definitions of the Calabi invariant in the case of compactly supported area preserving diffeomorphisms of the plane.
\end{rem}

In our definition of the action, we have chosen to normalise $\sigma$ by looking at the lower component of $\partial S$. The following result describes what happens on the upper component.
\begin{prop}
\label{salto}
Let $\Phi\in \mathcal{D}_L(S,\omega)$ and let $\sigma: S \rightarrow \R$ be its action. 
Let $\delta_x$ be a smooth path in $\partial S$ going from $(x,\pi)$ to $\Phi(x,\pi)$. Then
\[
\sigma(x,\pi) = \int_{\delta_x} \lambda + \mathrm{FLUX}(\Phi).
\]
\end{prop}

\begin{proof}
The same argument used in the paragraph above Definition \ref{action-defn} shows that it is enough to check the formula for $x=0$. In this case, by integrating over the path $\alpha_0:[0,\pi] \rightarrow S$, $\alpha_0(t):=(0,t)$, we find by Stokes theorem 
\[
\begin{split}
\sigma(0,\pi) &= \sigma(0,0) + \int_{\alpha_0} d\sigma = \int_{\gamma_0} \lambda -{\rm FLUX}(\Phi) + \int_{\alpha_0} ( \Phi^* \lambda - \lambda) \\ &= \int_{\gamma_0} \lambda -{\rm FLUX}(\Phi) + \int_{\Phi(\alpha_0)} \lambda + \int_{\alpha_0^{-1}} \lambda = \int_{\delta_0} \lambda -{\rm FLUX}(\Phi) + \iint_R h^*(d\lambda),
\end{split}
\] 
where $h: R \rightarrow S$ is a smooth map on a closed rectangle $R$ whose restriction to the boundary is given by the concatenation $\gamma_0 \# (\Phi\circ \alpha_0) \# \delta_0^{-1} \# \alpha_0^{-1}$. By using again Stokes theorem with the primitive $x\sin y\, dy$ of $\omega=d\lambda$, we get
\[
\iint_R h^*(d\lambda) = \int_{\gamma_0 \# (\Phi\circ \alpha_0) \# \delta_0^{-1} \# \alpha_0^{-1}} x\sin y\, dy = \int_{\Phi(\alpha_0)} x \sin y\, dy.
\]
By Proposition \ref{flux-prop}, the latter quantity coincides with twice the flux of $\Phi$, and the conclusion follows.
\end{proof}

\subsection{Generating functions}

As it is well known, area-preserving self-diffeomorphisms of the strip which satisfy a suitable monotonicity condition can be represented in terms of a generating function. See for instance \cite[Chapter 9]{ms98}. Here we need to review these facts in the case of diffeomorphims preserving the special two-form $\omega= \sin y \, dx\wedge dy$.

\begin{defn}
\label{mon-defn}
The diffeomorphism $\Phi=(X,Y)$ in $\mathcal{D}_L(S,\omega)$ is said to be monotone if $D_2 Y (x,y) >0$ for every $(x,y)\in S$.
\end{defn}

Assume that $\Phi=(X,Y)\in \mathcal{D}_L(S,\omega)$ is a monotone map. Then for every $x\in \R$ the map $y \mapsto Y(x,y)$ is a diffeomorphism of $[0,\pi]$ onto itself, and hence the map
\[
\Psi: S \rightarrow S, \qquad \Psi(x,y) = \bigl( x,Y(x,y) \bigr)
\]
is a diffeomorphism. Denoting by $y$ the second component of the inverse of $\Psi$, we can work with coordinates $(x,Y)$ on $S$ and consider the one-form
\[
\eta(x,Y) = ( \cos Y - \cos y)\, dx + (X-x) \sin Y\, dY \qquad \mbox{on } S.
\]
From the fact that $\Phi$ preserves $\omega$ we find
\[
\begin{split}
d\eta &= \sin Y \, dx \wedge dY - \sin y \, dx\wedge dy + \sin Y\, dX \wedge dY - \sin Y \, dx \wedge dY \\&= -\sin y \, dx \wedge dy + \sin Y \, dX \wedge dY = 0,
\end{split}
\]
so $\eta$ is closed. Let $W=W(x,Y)$ be a primitive of $\eta$. Then also $(x,y) \mapsto W(x+L,y)$ is a primitive of $\eta$, and hence
\[
W(x+L,Y) - W(x,Y) = c, \qquad \forall  (x,Y)\in S,
\]
for some real number $c$. Since the integral of $\eta$ on any path in $\partial S$ connecting $(0,0)$ to $(L,0)$ vanishes, the constant $c$ must be zero, and hence any primitive $W$ of $\eta$ is $L$-periodic.
By writing 
\[
dW (x,Y) = D_1 W(x,Y)\, dx + D_2W (x,Y)\, dY,
\]
and using the definition of $\eta$, we obtain the following:

\begin{prop}
\label{gen-funct-prop} 
Assume that $\Phi$ in $\mathcal{D}_L(S,\omega)$ is a monotone map.  Then there exists a smooth function $W: S \rightarrow \R$ such that the following holds: $\Phi(x,y)=(X,Y)$ if and only if
\begin{eqnarray}
\label{gen-eq-1}
(X-x) \sin Y & = & D_2 W(x,Y), \\
\label{gen-eq-2}
\cos Y - \cos y & = & D_1 W(x,Y). 
\end{eqnarray}
The function $W$ is $L$-periodic in the first variable. It is uniquely defined up to the addition of a real constant.
\end{prop}

A function $W$ as above is called a generating function of $\Phi$.
Equation (\ref{gen-eq-2}) implies that $W$ is constant on each of the two connected components of the boundary of $S$. The difference between these two constant values coincides with twice the flux of $\Phi$:

\begin{prop}
\label{gen-flux-prop}
If $W$ is a generating function of the monotone map $\Phi\in \mathcal{D}_L(S,\omega)$, then
\[
\mathrm{FLUX}(\Phi) = \frac{1}{2} \left( W|_{\R \times \{\pi\}} - W|_{\R \times \{0\}} \right).
\]
\end{prop}

\begin{proof}
By Proposition \ref{flux-prop} and (\ref{gen-eq-1}) we compute
\[
\begin{split}
\mathrm{FLUX}(\Phi) &= \frac{1}{2} \int_{\Phi(\alpha_0)} x \sin y\, dy = \frac{1}{2} \int_{\alpha_0} X \sin Y\, dY = \frac{1}{2} \int_{\alpha_0} (X-x) \sin Y\, dY \\ &= \frac{1}{2} \int_{\alpha_0} D_2 W(x,Y)\, dY = \frac{1}{2} \left( W|_{\R \times \{\pi\}} - W|_{\R \times \{0\}} \right),
\end{split}
\]
where we have used the fact that $x=0$ on the path $\alpha_0$ which is defined in Proposition~\ref{flux-prop}.
\end{proof}

By the above proposition, we can choose the free additive constant of the generating function $W$ in such a way that:
\begin{equation}\label{norW}
W|_{\R\times \{0\}} = -\mathrm{FLUX}(\Phi), \qquad W|_{\R\times \{\pi\}} = \mathrm{FLUX}(\Phi).
\end{equation}
We conclude this section by expressing the action and the Calabi invariant of a monotone element of $\mathcal{D}_L(S,\omega)$ in terms of its generating function, normalised by the above condition.

\begin{prop}
\label{gen-cal-prop}
Let $\Phi=(X,Y)\in \mathcal{D}_L(S,\omega)$ be a monotone map, and denote by $W$ the generating function of $\Phi$ normalised by (\ref{norW}). Then we have:
\begin{enumerate}[(i)]
\item The action of $\Phi$ is the function
\[
\sigma(x,y) = W(x,Y(x,y)) + D_2 W(x,Y(x,y)) \cot Y(x,y).
\]
\item If moreover $\mathrm{FLUX}(\Phi)=0$, then the Calabi invariant of $\Phi$ is the number
\[
\mathrm{CAL}(\Phi)= \frac{1}{2L} \iint_{[0,L]\times [0,\pi]} \bigl( W(x,y) + W(x,Y(x,y)) \bigr)\, \omega(x,y).
\]
\end{enumerate}
\end{prop}

The formula for $\sigma$ in (i) is valid only in the interior of $S$, because the cotangent function diverges at $0$ and $\pi$. Since $D_2 W$ vanishes on the boundary of $S$, thanks to (\ref{gen-eq-1}), this formula defines a smooth function on $S$ by setting
\[
\sigma(x,0) =  W(x,0) + D_{22} W(x,0), \qquad \sigma(x,\pi) = W(x,\pi) + D_{22} W(x,\pi),
\]
for every $x\in \R$.

\begin{proof}
Let us check that the function $\sigma$ which is defined in (i) coincides with the action of $\Phi$. By (\ref{gen-eq-1}) we have
\begin{equation}
\label{comoda}
\sigma = W + D_2 W \cot Y = W +  (X-x) \cos Y
\end{equation}
on $\mathrm{int}(S)$.
By continuity, this formula for $\sigma$ is valid on the whole $S$.
By differentiating it and using again (\ref{gen-eq-1}) together with (\ref{gen-eq-2}), we obtain
\[
\begin{split}
d\sigma &= dW -  (X-x) \sin Y \, dY + \cos Y (dX-dx) \\ &= dW - D_2 W\, dY  + \cos Y (dX-dx) = D_1 W \, dx + \cos Y (dX-dx) \\ &= (\cos Y- \cos y)\, dx + \cos Y (dX-dx) = \cos Y\, dX - \cos y\, dx = \Phi^* \lambda - \lambda.
\end{split}
\] 
Therefore, $\sigma$ satisfies (\ref{action1}). Evaluating (\ref{comoda}) in $(0,0)$ we find
\[
\sigma(0,0) = W(0,0) + X(0,0) = - \mathrm{FLUX}(\Phi) + X(0,0) = - \mathrm{FLUX}(\Phi) + \int_{\gamma_0} \lambda,
\]
where $\gamma_0$ is a path in $\partial S$ going from $(0,0)$ to $\Phi(0,0)$. We conclude that $\sigma$ satisfies also (\ref{action2}), and hence coincides with the action of $\Phi$. This proves (i).

We now use (i) in order to compute the integral of the two form $\sigma\, \omega$ on $[0,L]\times [0,\pi]$. We start from the identity
\begin{equation}\label{cal1}
\begin{split}
\iint_{[0,L]\times [0,\pi]} \sigma\, \omega = & \iint_{[0,L]\times [0,\pi]} W(x,Y(x,y)) \, \omega(x,y) \\ &+ \iint_{[0,L]\times [0,\pi]} D_2W(x,Y(x,y)) \cot Y(x,y) \sin y \, dx \wedge dy,
\end{split}
\end{equation}
and we manipulate the last integral. By differentiating (\ref{gen-eq-2}), that is, the identity
\[
\cos Y(x,y) - \cos y = D_1 W(x,Y(x,y)),
\]
we obtain 
\[
\sin y \, dy = \sin Y \, dY + D_{11} W \, dx + D_{12} W \, dY.
\]
By the above formula, the integrand in the last integral in (\ref{cal1}) can be rewritten as
\begin{equation}
\label{2forms}
\begin{split}
D_2W \cot Y \sin y \, dx \wedge dy &= D_2W \cot Y  \, dx \wedge ( \sin Y \, dY  + D_{12} W\, dY) \\ &= D_2 W\cos Y   \, dx\wedge dY + D_2 W D_{12} W \cot Y\, dx\wedge dY.
\end{split}
\end{equation}
We integrate the above two forms separately. By the $L$-periodicity in $x$, the integral of the first two-form can be manipulated as follows: 
\begin{equation}
\label{cal3}
\begin{split}
& \iint_{[0,L]\times [0,\pi]}  D_2 W(x,Y(x,y)) \cos Y(x,y)\, dx\wedge dY(x,y) \\ &\qquad=  \iint_{[0,L]\times [0,\pi]}   D_2 W(x,Y) \cos Y\, dx\wedge dY \\ &\qquad= \int_0^L \left( \int_0^\pi   D_2 W(x,Y) \cos Y \, dY \right)\, dx \\ &\qquad= \int_0^L \left( \Bigl[  W(x,Y)\cos Y  \Bigr]_{Y=0}^{Y=\pi} + \int_0^{\pi} W(x,Y)\sin Y\, dY \right)\, dx \\&\qquad= - L \bigl( W|_{\R \times \{\pi\}} + W|_{\R \times \{0\}} \bigr) + \iint_{[0,L]\times [0,\pi]} W(x,Y) \sin Y \, dx\wedge dY \\ &\qquad= - L \bigl( - \mathrm{FLUX}(\Phi) + \mathrm{FLUX}(\Phi) \bigr) + \iint_{[0,L]\times [0,\pi]} W(x,y) \, \sin y \, dx\wedge dy \\ &\qquad= \iint_{[0,L]\times [0,\pi]} W(x,y) \, \omega(x,y),
\end{split}
\end{equation}
where we have used the normalization condition~\eqref{norW}. The integral of the second form in the right-hand side of (\ref{2forms}) vanishes, because
\begin{equation}
\label{cal4}
\begin{split}
\iint_{[0,L]\times [0,\pi]} & D_2 W D_{12} W \cot Y \, dx\wedge dY \\ &= \frac{1}{2} \iint_{[0,L]\times [0,\pi]} D_1 (D_2 W)^2 \cot Y \, dx \wedge dY \\ &= \frac{1}{2} \int_0^{\pi} \cot Y \left( \int_0^L D_1 (D_2 W)^2\, dx \right)\, dY = 0,
\end{split}
\end{equation}
by $L$-periodicity in $x$. By (\ref{cal1}), (\ref{2forms}), (\ref{cal3}) and (\ref{cal4}) we obtain
\[
\iint_{[0,L]\times [0,\pi]} \sigma\, \omega = \iint_{[0,L]\times [0,\pi]} \bigl( W(x,Y(x,y)) + W(x,y) \bigr)\, \omega(x,y),
\]
and (ii) follows.
\end{proof}

\subsection{The Calabi invariant and the action at fixed points}

We are now in the position to prove the main result of this first part.

\begin{thm}\label{thm_calabi}
Let $\Phi$ be a monotone  element of $\mathcal{D}_L(S,\omega)$ which is different from the identity and has zero flux. If $\mathrm{CAL}(\Phi)\leq 0$ (resp.\ $\mathrm{CAL}(\Phi)\geq 0$), then $\Phi$ has an interior fixed point with negative (resp.\ positive) action.
\end{thm}

\begin{proof}
Let $W$ be the generating function of $\Phi$ normalised by the condition (\ref{norW}). Since $\Phi$ has zero flux, this condition says that $W$ is zero on the boundary of $S$.
Since $\Phi$ is not the identity, $W$ is not identically zero. Then the condition $\mathrm{CAL}(\Phi)\leq 0$ and the formula of Proposition \ref{gen-cal-prop} (ii) for $\mathrm{CAL}(\Phi)$ imply that $W$ is somewhere negative. Being a continuous periodic function, $W$ achieves its minimum at some interior point $(x,Y)\in \mathrm{int}(S)$. Since the differential of $W$ vanishes at $(x,Y)$, equations (\ref{gen-eq-1}) and (\ref{gen-eq-2}) imply that $(x,y):=(x,Y)$ is a fixed point of $\Phi$. By Proposition \ref{gen-cal-prop} (i),
\[
\sigma(x,y) = W(x,Y) < 0.
\]
Therefore, $(x,y)$ is an interior fixed point of $\Phi$ with negative action. The case $\mathrm{CAL}(\Phi)\geq 0$ is completely analogous.
\end{proof}

\section{The geodesic flow on a positively curved two-sphere}\label{sec_geodesic_flow_pos_curv}

Throughout this section, a smooth oriented Riemannian two-sphere $(S^2,g)$ is fixed. The associated unit tangent bundle is 
\[
T^1 S^2:=\{v\in TS^2 \mid g_{\pi(v)}(v,v)=1\},
\] 
where $\pi:TS^2 \to S^2$ denotes the bundle projection.  For each $v \in T^1 S^2$, we denote by $v^\perp\in T_{\pi(v)}S^2$ the unit vector perpendicular to $v$ such that $\{v,v^\perp\}$ is a positive basis of $T_{\pi(v)}S^2$.

We shall deal always with Riemannian metrics $g$ having positive Gaussian curvature $K$ and shall often use Klingenberg's lower bound on the injectivity radius $\mathrm{inj}(g)$ of the metric $g$ from \cite{kli59}, that is,
\begin{equation}\label{injradius-estimate}
\mathrm{inj}(g) \geq \frac{\pi}{\sqrt{\max K}},
\end{equation}
see also \cite[Theorem 2.6.9]{kli82}.

\subsection{Extension and regularity of the Birkhoff map}\label{sec_birkhoff_map}
\label{extsec}

Let $\gamma:\R / L\Z \to S^2$ be a simple closed geodesic of length $L$ parametrised by arc-length, i.e. satisfying $g_{\gamma}(\dot \gamma,\dot \gamma)\equiv1$. The smooth unit vector field $\dot \gamma^\perp$ along $\gamma$ determines the Birkhoff annuli 
\begin{equation}\label{birkhoff_annuli}
\begin{aligned}
&\Sigma_\gamma^+ := \{\cos y \ \dot \gamma(x) + \sin y \ \dot \gamma^\perp(x)\in T^1S^2 \mid (x,y)\in \R / L\Z \times [0,\pi]\}, \\
&\Sigma_\gamma^- := \{\cos y \ \dot \gamma(x) + \sin y \ \dot \gamma^\perp(x)\in T^1S^2 \mid (x,y)\in \R / L\Z \times [-\pi,0]\}.
\end{aligned}
\end{equation}
These sets are embedded closed annuli and $(x,y)$ are smooth coordinates on them. The annuli $\Sigma_{\gamma}^+$ and $\Sigma_{\gamma}^-$ intersect along their boundaries $\partial \Sigma_{\gamma}^+=\partial \Sigma_{\gamma}^-$. This common boundary has two components, one containing  unit vectors $\dot \gamma$ and the other containing unit vectors $-\dot \gamma$.  We denote the open annuli by 
\[
\mathrm{int}(\Sigma^+_\gamma) := \Sigma^+_\gamma \setminus \partial\Sigma^+_\gamma, \qquad \mathrm{int}(\Sigma^-_\gamma) := \Sigma^-_\gamma \setminus \partial\Sigma^-_\gamma. 
\]

Let $\phi_t$ be the geodesic flow on $T^1S^2$. We define the functions
\[
\begin{split}
\tau_+ : \mathrm{int}(\Sigma_{\gamma}^+) \rightarrow (0,+\infty], \qquad \tau_+(v) := \inf \{ t>0 \mid \phi_t(v) \in \mathrm{int}( \Sigma^-_\gamma) \}, \\
\tau_- : \mathrm{int}(\Sigma_{\gamma}^-) \rightarrow (0,+\infty], \qquad \tau_-(v) := \inf \{ t>0 \mid \phi_t(v) \in \mathrm{int}( \Sigma^+_\gamma) \},
\end{split}
\]
where the infimum of the empty set is $+\infty$. The functions $\tau_+$ and $\tau_-$ are the transition times to go from the interior of $\Sigma_{\gamma}^+$ to the interior of $\Sigma_{\gamma}^-$ and the other way round. The first return time to $\Sigma_{\gamma}^+$ is instead the function
\[
\tau : \mathrm{int}(\Sigma_{\gamma}^+) \rightarrow (0,+\infty], \qquad \tau(v) := \inf \{ t>0 \mid \phi_t(v) \in \mathrm{int}( \Sigma^+_\gamma) \}.
\]
Recall the following celebrated theorem due to Birkhoff (see also \cite{ban93}):

\begin{thm}[Birkhoff~\cite{bir27}]\label{thm_Birkhoff}
If the Gaussian curvature of $g$ is everywhere positive then the functions $\tau_+$, $\tau_-$ and $\tau$ are everywhere finite.
\end{thm}

Thanks to the above result, we have the transition maps
\[
\begin{split}
\varphi_+ : \mathrm{int}(\Sigma_{\gamma}^+) \rightarrow \mathrm{int}(\Sigma_{\gamma}^-), \qquad \varphi_+(v) := \phi_{\tau_+(v)}(v), \\
\varphi_- : \mathrm{int}(\Sigma_{\gamma}^-) \rightarrow \mathrm{int}(\Sigma_{\gamma}^+), \qquad \varphi_-(v) := \phi_{\tau_-(v)}(v),
\end{split}
\]
and the first return map 
\[
\varphi: \mathrm{int}(\Sigma_{\gamma}^+) \rightarrow \mathrm{int}(\Sigma_{\gamma}^+), \qquad \varphi(v) := \phi_{\tau(v)}(v).
\]
By construction,
\begin{eqnarray}
\label{comp-T}
\varphi &=& \varphi_- \circ \varphi_+, \\
\label{comp-tau}
\tau &=& \tau_+ + \tau_-\circ \varphi_+.
\end{eqnarray}

Using the implicit function theorem and the fact that the geodesic flow is transverse to both $\mathrm{int}(\Sigma^+_\gamma)$ and $\mathrm{int}(\Sigma^-_\gamma)$, one easily proves that the functions $\tau^+$, $\tau_-$ and $\tau$ are smooth. These functions have smooth extensions to the closure of their domains. More precisely, we have the following statement.

\begin{prop}\label{prop_smoothness_return_time}
Assume that the Gaussian curvature of $(S^2,g)$ is everywhere positive. Then:
\begin{enumerate}[(i)]
\item The functions $\tau_+$ and $\tau_-$ can be smoothly extended to $\Sigma^+_\gamma$ and $\Sigma_{\gamma}^-$, respectively, as follows: $\tau_+(\dot\gamma(x))=\tau_-(\dot\gamma(x))$ is the time to the first conjugate point along the geodesic ray $t \in [0,+\infty) \mapsto\gamma(x+t)$, and $\tau_+(-\dot\gamma(x))=\tau_-(-\dot\gamma(x))$ is the time to the first conjugate point along the geodesic ray $t\in [0,+\infty) \mapsto \gamma(x-t)$. 
\item The function $\tau$ can be smoothly extended to $\Sigma^+_\gamma$ as follows: $\tau(\dot\gamma(x))$ is the time to the second conjugate point along the geodesic ray $t \in [0,+\infty) \mapsto\gamma(x+t)$, and $\tau(-\dot\gamma(x))$ is the time to the second conjugate point along the geodesic ray $t\in [0,+\infty) \mapsto\gamma(x-t)$. 
\end{enumerate}
\end{prop}
The smooth extensions of $\tau^+$, $\tau_-$ and $\tau$ are denoted by the same symbols. 
The above proposition has the following consequence:

\begin{cor}\label{cor_smoothness_return_map}
Suppose that the Gaussian curvature of $(S^2,g)$ is everywhere positive. Then the formulas
\[
v\mapsto \phi_{\tau_+(v)}(v), \qquad v\mapsto \phi_{\tau_-(v)}(v) \quad \mbox{and} \quad v\mapsto \phi_{\tau(v)}(v)
\]
define smooth extensions of the maps $\varphi_+$, $\varphi_-$ and $\varphi$ to diffeomorphisms
\[
\varphi_+ : \Sigma_{\gamma}^+ \rightarrow \Sigma_{\gamma}^-, \qquad \varphi_- : \Sigma_{\gamma}^- \rightarrow \Sigma_{\gamma}^+ \quad \mbox{and} \quad \varphi : \Sigma_{\gamma}^+ \rightarrow \Sigma_{\gamma}^+,
\]
which still satisfy (\ref{comp-T}) and (\ref{comp-tau}).
\end{cor}

\begin{proof}
The smoothness of the geodesic flow $\phi$ and of the functions $\tau_+$, $\tau_-$ and $\tau$ imply that $\varphi_+$, $\varphi_-$ and $\varphi$ are smooth. Since the inverses of these maps on the interior of their domains have analogous definitions, such as for instance
\[
\varphi_+^{-1} (v) = \phi_{\hat\tau_+(v)}(v), \qquad \mbox{where} \qquad \hat\tau_+(v) := \sup\{ t<0 \mid\phi_t(v) \in \mathrm{int}(\Sigma_{\gamma}^+) \},
\]
the maps $\varphi^{-1}_+$, $\varphi_-^{-1}$ and $\varphi^{-1}$ have also smooth extensions to the closure of their domains, and hence $\varphi_+$, $\varphi_-$ and $\varphi$ are diffeomorphisms.
\end{proof}

For sake of completeness, we include a proof of Proposition \ref{prop_smoothness_return_time}. A proof of statement (ii) has recently appeared in \cite{sch14}. This proof is based on a technical lemma about return time functions of a certain class of flow, which we now introduce.
Consider coordinates $(x,q,p) \in \R/\Z \times \R^2$ and a smooth tangent vector field $X$ on $\R/\Z \times \R^2$ satisfying
\begin{equation}\label{periodic_orbit}
X(x,0,0)=(1,0,0), \qquad \forall x\in \R/\Z.
\end{equation}
If we denote by $\psi_t$ the flow of $X$ then 
\[
\psi_t(x,0,0)=(x+t,0,0), \qquad  \forall x\in \R/\Z,
\]
and $P:=\R/\Z\times0$ is a $1$-manifold invariant by the flow. We assume also that for every $x\in\R/\Z$ and $t\in\R$ the subspace $\{0\}\times \R^2\subset \R^3$ is preserved by the differential of the flow, i.e.
\begin{equation}\label{linear_flow_1}
D\psi_t(x,0,0) \bigl[ \{0\} \times\R^2 \bigr]=\{0\} \times \R^2, \qquad \forall x\in \R/\Z, \; \forall t\in \R.
\end{equation}

For each $\delta\in(0,\infty]$ consider the annuli
\[
A_\delta^+ := \R/\Z\times [0,\delta), \qquad A_\delta^- := \R/\Z\times (-\delta, 0],
\]
both equipped with the coordinates $(x,y)$. To each point $(x,y)\in \mathrm{int}( A_\delta^+)$ one may try to associate the point $\varphi_+(x,y)\in \mathrm{int}( A_\delta^-)$ given by the formula
\begin{equation}
\varphi_+(x,y) = \psi_{\tau_+(x,y)}(x,y,0)
\end{equation}
where $\tau_+(x,y)$ is a tentative ``first hitting time of $A_{\delta}^-$'', that is,
\begin{equation}
\tau_+(x,y) = \inf \ \bigl\{t>0 \; \mid \; \psi_t(x,y,0) \in \mathrm{int}(A_\infty^-)\times \{0\}\bigr\}.
\end{equation}
Of course, in general $\tau_+$ and $\varphi_+$ may not be well-defined, even for small $\delta$. Our purpose below is to give a sufficient condition on the vector field $X$ to guarantee that, if $\delta$ is small enough, $\tau_+$ and $\varphi_+$ are well-defined smooth functions on $\mathrm{int}( A_\delta^+)$ which extend smoothly to $A_\delta^+$. In the following  definition and in the proof of the lemma below, we identify $\R^2$ with $\C$.

\begin{defn}
Fix some $x\in \R/\Z$ and $v\in \R^2\setminus \{0\}$. By (\ref{linear_flow_1}) the image of $(0,v)$ by the differential of $\psi_t$ at $(x,0,0)$ has the form  
\[
D\psi_t(x,0,0)\bigl[(0,v)\bigr]=(0,\rho(t)e^{i\theta(t)}),
\] 
for suitable smooth functions $\rho>0$ and $\theta$, where $\rho$ is unique and $\theta$ is unique up to the addition of an integer multiple of $2\pi$. We say that the linearised flow along $P$ has a positive twist if for every choice of $x\in \R/\Z$ and $v\in \R^2\setminus \{0\}$ the function $\theta$ which is defined above satisfies $\theta'(t)>0 $ for all $t\in \R$.
\end{defn}

\begin{lem}\label{lemma_return_time}
If the linearised flow along $P$ has a positive twist, then there exists $\delta_0>0$ such that $\tau_+$ is a well-defined smooth function on $\mathrm{int}( A_{\delta_0}^+)$ which extends smoothly as a positive function on $A_{\delta_0}^+$. Moreover, this extension is described by the formula
\begin{equation}\label{tau_boundary}
\tau^+(x,0) = \inf \ \{ t>0 \; \mid \; D\psi_t(x,0,0)[\partial_y] \in \R^- \partial_y \},
\end{equation}
where $\partial_y := (0,1,0)$.
\end{lem}

\begin{proof}
Write $w=y+iz$ and $Y=X_2+iX_3$, where $(X_1,X_2,X_3)$ are the components of the vector field $X$. Then 
\[
X(x,w)=\bigl(X_1(x,w),Y(x,w)\bigr).
\] 
By~\eqref{periodic_orbit} we have $X_1(x,0)=1$ and $Y(x,0)=0$. Consider $W(x,w) \in \L_\R(\C)$ defined by
\[
W(x,w)=\int_0^1 D_2Y(x, s w)\, ds,
\]
where $D_2Y$ denotes derivative with respect to the second variable. Then
\[
W(x,0)=D_2Y(x,0), \qquad Y(x,w)=W(x,w)w. 
\]

We shall now translate the assumption that the linearised flow along $P$ has a positive twist into properties of $W(x,0)$. Choose $v_0\in \C\setminus0$. Using~\eqref{linear_flow_1} we find a smooth non-vanishing complex valued function $v$ such that 
\[
D\psi_t(x,0)[(0,v_0)]=(0,v(t)).
\] 
From 
\[
\frac{d}{dt}D\psi_t = (DX\circ \psi_t) D\psi_t,
\] 
and from~\eqref{linear_flow_1} we get the linear ODE
\[
\dot v(t) = D_2Y(x+t,0) v(t) = W(x+t,0)v(t).
\]
Writing $v(t)=r(t)e^{i\theta(t)}$ with smooth functions $r>0$ and $\theta$, we know that
\begin{equation}\label{linear_flow_2}
\begin{split}
\theta' &= \mathrm{Re}\, \left( \frac{\dot v}{iv} \right) = \mathrm{Re}\, \left( \frac{W(x+t,0)v}{iv} \ \frac{\overline{iv}}{\overline{iv}} \right) \\ &= \frac{\left< W(x+t,0)v,iv \right>}{|v|^2} = \left< W(x+t,0)e^{i\theta},ie^{i\theta} \right>,
\end{split}
\end{equation}
where $\langle \cdot,\cdot \rangle$ denotes the Hermitian product on $\C$.
Since $x,t$ and $v(t)$ can take arbitrary values, we conclude from the above formula and the assumptions of the lemma that
\begin{equation}\label{positivity_1}
\left< W(x,0)u,iu \right>>0, \qquad \forall u\in \C\setminus \{0\}, \; \forall x\in\R/\Z.
\end{equation}

Consider polar coordinates $(r,\theta)\in[0,+\infty)\times\R/2\pi\Z$ in the $w$-plane given by $w=y+iz=re^{i\theta}$. The map
\[
(x,r,\theta) \mapsto X(x,re^{i\theta})
\]
is smooth. Using the formulas 
\[
\partial_y=\frac{y}{r}\partial_r-\frac{z}{r^2}\partial_\theta, \qquad \partial_z=\frac{z}{r}\partial_r+\frac{y}{r^2}\partial_\theta,
\] 
we obtain that the vector field $X$ pulls back by this change of coordinates to a smooth vector field
\[
Z=(Z_1,Z_2,Z_3),
\]
which is given by
\begin{equation}
\left\{
\begin{aligned}
Z_1(x,r,\theta) &= X_1(x,re^{i\theta}), \\
Z_2(x,r,\theta) &= \cos\theta \ X_2(x,re^{i\theta}) + \sin\theta \ X_3(x,re^{i\theta}), \\
Z_3(x,r,\theta) &= \frac{1}{r} \left( \cos\theta \ X_3(x,re^{i\theta}) - \sin\theta \ X_2(x,re^{i\theta}) \right).
\end{aligned}
\right.
\end{equation}
Indeed, the smoothness of $Z_1$ and $Z_2$ follows immediately from the above formulas, while that of $Z_3$ needs a little more care. Since $X_2,X_3$ vanish on $\R/\Z \times \{0\}$, we can find smooth functions $X_{2,2},X_{2,3},X_{3,2},X_{3,3}$ such that
\[
\begin{split}
X_2(x,y+iz) &=y X_{2,2}(x,y+iz)+z X_{2,3}(x,y+iz), \\
X_3(x,y+iz)&=y X_{3,2}(x,y+iz)+z X_{3,3}(x,y+iz),
\end{split}
\]
where
\[
\begin{split}
X_{2,2}(x,0)= D_2 X_2 (x,0,0), &\qquad X_{2,3}(x,0)=D_3 X_2 (x,0,0), \\
X_{3,2}(x,0)= D_2 X_3 (x,0,0), &\qquad X_{3,3}(x,0)=D_3 X_3(x,0,0),
\end{split}
\]
and
\[
W(z,w)=\begin{bmatrix} X_{2,2}(x,w) & X_{2,3}(x,w) \\ X_{3,2}(x,w) & X_{3,3}(x,w) \end{bmatrix}.
\]
Substituting $y=r\cos\theta$, $z=r\sin\theta$ we find 
\begin{equation}
Z_3(x,r,\theta) = \left< W(x,re^{i\theta})e^{i\theta},ie^{i\theta}\right>.
\end{equation}
Thus $Z_3$ is a smooth function of $(x,r,\theta)$ and
\begin{equation}\label{positivity_2}
Z_3(x,0,\theta)>0, \qquad \forall x\in \R/\Z, \; \forall \theta \in \R/2\pi \Z,
\end{equation}
thanks to~\eqref{positivity_1}.

From now on we lift the variable $\theta$ from $\R/2\pi\Z$ to the universal covering $\R$ and think of the vector field $Z$ as a smooth vector field defined on $\R/\Z\times [0,+\infty)\times\R$, having components $2\pi$-periodic in $\theta$. Clearly this vector field is tangent to $\{r=0\}$.

Let $\zeta_t$ denote the flow of $Z$. After changing coordinates and lifting, we see that the conclusions of the lemma will follow if we check that
\begin{equation}
\tau_+(x,r) = \inf \{t>0 \mid \theta\circ \zeta_t(x,r,0) =\pi\}
\end{equation}
defines a smooth function of $(x,r)\in \R/\Z\times [0,\delta)$ when $\delta$ is small enough. By~\eqref{positivity_2} we see that if $\delta_0$ is fixed small enough then $\tau_+(x,r)$ is a well-defined, uniformly bounded and strictly positive function of $(x,r)\in \R/\Z\times [0,\delta_0)$. Here we used that $Z$ is tangent to $\{r=0\}$. Perhaps after shrinking $\delta_0$, we may also assume that 
\begin{equation}
\label{pos}
Z_3(\zeta_t(x,r,0))>0, \qquad \forall (x,r)\in\R/\Z\times [0,\delta_0), \; \forall t\in[0,\tau_+(z,r)].
\end{equation}
Continuity and smoothness properties of $\tau_+$ remain to be checked. This is achieved with the aid of the implicit function theorem. In fact, consider the smooth function
\[
F: \R\times\R/\Z\times[0,+\infty) \rightarrow \R, \qquad
F(\tau,x,r) := \theta \circ \zeta_{\tau}(x,r,0).
\]
Since 
\[
D_1F(\tau,x,r)=d\theta \bigl[ Z(\zeta_{\tau}(x,r,0)) \bigr] = Z_3 (\zeta_{\tau}(x,r,0)), 
\]
it follows from (\ref{pos}) and from the implicit function theorem that the equation 
\[
F(\tau_+,x,r)=\pi
\]
determines $\tau_+=\tau_+(x,r)$ as a smooth function of $(x,r)\in\R/\Z\times[0,\delta_0)$.

We now check formula~\eqref{tau_boundary} for $\tau_+(x,0)$. From the above equations one sees that $\theta(t)=\theta\circ\zeta_t(x,0,0)$ satisfies the differential equation 
\[
\theta'(t)=\left<D_2Y(x+t,0)e^{i\theta},ie^{i\theta}\right>,
\] 
with initial condition $\theta(0)=0$. Thanks to~\eqref{linear_flow_2}, this is exactly the same initial value problem for the argument $\hat{\theta}(t)$ of the solution $v(t)=\rho(t)e^{i\hat{\theta}(t)}$ of the linearised flow starting at the base point $(x,0)$ applied to the vector $\partial_y$.
\end{proof}

In order to prove Proposition \ref{prop_smoothness_return_time}, it is enough to show that coordinates can be arranged in such a way that the geodesic flow near a simple closed geodesic $\gamma$ meets the assumptions of Lemma~\ref{lemma_return_time} when the Gaussian curvature is positive along $\gamma$. We will assume for simplicity, and without loss of generality, that $L=1$. We start by recalling basic facts from Riemannian geometry and fixing some notation.

Given $v\in TS^2$, let $\V_v\subset T_vTS^2$ be the vertical subspace, which is defined as $\V_v:=\ker d\pi(v)$. The isomorphism \[
i_{\V_v}:T_{\pi(v)}S^2 \to \V_v
\] 
is defined as 
\[
i_{\V_v}(w):=\frac{d}{dt}(v+tw)\Bigr|_{t=0}, \qquad \forall w\in T_{\pi(v)} S^2.
\] 
The Levi-Civita connection of $g$ determines a bundle map $K:TTS^2\to TS^2$ satisfying $\nabla_YX=K(dX \circ Y)$, where $X,Y$ are vector fields on $S^2$ seen as maps $S^2\to TS^2$. The horizontal subspace $\h_v:=\ker K|_{T_vTS^2}$ satisfies $T_v TS^2 = \V_v \oplus \h_v$. There is an isomorphism 
\[
i_{\h_v}:T_{\pi(v)}S^2 \to \h_v, \qquad i_{\h_v}(w) := \frac{d}{dt} V(t)\Bigr|_{t=0}, \qquad \forall w\in T_{\pi(v)} S^2,
\]
where $V$ is the parallel vector field along the geodesic $\beta(t)$ satisfying $\dot \beta(0)=w$ with initial condition $V(0)=v$, seen as a curve in $TTS^2$. The isomorphism $i_{\h_v}$ satisfies
\begin{equation}\label{horiz}
d\pi(v)\bigl[i_{\h_v}(w)\bigr]=w, \qquad \forall w\in T_{\pi(v)} S^2.
\end{equation}
For each $v\in T^1S^2$ we have $$T_vT^1S^2 = {\rm span}\{i_{\V_v}(v^\perp), i_{\h_v}(v^\perp),i_{\h_v}(v)\}.$$ The Hilbert form $\lambda_H$ on $TS^2$ is given by 
\begin{equation}\label{hilb}
\lambda_H(v) [\zeta] := g_{\pi(v)}\bigl(v, d\pi (v) [\zeta]\bigr), \qquad \forall \zeta\in T_vS^2,
\end{equation} 
and restricts to a contact form $\alpha$ on $T^1S^2$. The contact structure $\xi := \ker \alpha$ is trivial since $$\xi_v = {\rm span}\{i_{\V_v}(v^\perp),i_{\h_v}(v^\perp)\}.$$ The Reeb vector field $R_\alpha$ of $\alpha$ coincides with $i_{\h_v}(v)$, and $\{i_{\V_v}(v^\perp),i_{\h_v}(v^\perp)\}$ forms a symplectic basis for $d\alpha|_{\xi_v}$, because 
\[
d\alpha(v)\bigl[ i_{\V_v}(v^\perp),i_{\h_v}(v^\perp)\bigr]=1.
\] 
If $(x,y)$ are the standard coordinates on $\Sigma_{\gamma}^{\pm}$ given by
\[
v = \cos y\ \dot\gamma(x) + \sin y \ \dot\gamma(x)^{\perp},
\]
then the tangent vectors $\partial_x$ and $\partial_y$ in $T_v \Sigma_{\gamma}^{\pm}$ are
\begin{equation}\label{vettori-tangenti}
\begin{split}
\partial_x &= i_{\h_v}(\dot \gamma(x))= \cos y\ i_{\h_v}(v) - \sin y\ i_{\h_v}(v^\perp), \\
\partial_y &= i_{\V_v}(v^\perp).
\end{split}
\end{equation}

\begin{proof}[Proof of Proposition~\ref{prop_smoothness_return_time}]
It is enough to prove statement (i) for the function $\tau_+$. In fact, the case of $\tau_-$ follows by inverting the orientation of $\gamma$, and statement (ii) is then a direct consequence of the identity (\ref{comp-tau}).

By (\ref{vettori-tangenti}) the vector field $R_\alpha=i_{\h_v}(v)$ is transverse to the interior of $\Sigma_\gamma^\pm$. The smooth vector field 
\[
i_{\h_v}(\dot \gamma^\perp)= \sin y \ i_{\h_v}(v) + \cos y\  i_{\h_v}(v^\perp)
\]
along $\Sigma_\gamma^+\cup \Sigma_\gamma^-$ is transverse to it near $\dot\gamma$. To obtain the desired coordinates near $\dot\gamma$ we proceed  as follows: let $\bar g$ be the Riemannian metric on $T^1S^2$ defined by 
\[
\bar g_v(\zeta_1,\zeta_2) := \alpha(\zeta_1)\alpha(\zeta_2) + d\alpha \bigl(\pi_\xi (\zeta_1), J \pi_\xi ( \zeta_2)\bigr),
\]
where $J:\xi \to \xi$ is the $d\lambda$-compatible complex structure determined by \[
J \bigl( i_{\V_v}(v^\perp) \bigr) =i_{\h_v}(v^\perp),
\]
$\pi_\xi:T^1S^2 \to \xi$ is the projection along $R_\alpha$, and $\zeta_1,\zeta_2 \in T_vT^1S^2$ are arbitrary.  Note that $\xi$ is orthogonal to $\R R_\alpha$ with respect to $\bar g$ and $\bar g(i_{\V_v}(v^\perp),i_{\h_v}(v^\perp))=0$.

Denote by ${\rm Exp}$ the exponential map of $\bar g$. Then for all $\delta>0$ sufficiently small, the map 
\[
\begin{aligned}
&\R/\Z \times (-\delta,\delta) \times (-\delta,\delta) \to \mathcal U \\
&(x,y,z) \mapsto {\rm Exp}_{v=\cos y \dot \gamma(x) + \sin y \dot \gamma^\perp(x)} \left( z (\sin y \ i_{\h_v}(v) + \cos y \ i_{\h_v}(v^\perp)) \right)
\end{aligned} 
\]
is a diffeomorphism, where $\mathcal{U}\subset T^1S^2$ is a small tubular neighborhood of $\dot \gamma$. In coordinates $(x,y,z)$, we have 
\begin{equation}\label{identi}
\begin{aligned}
\dot \gamma &\equiv \R/\Z \times \{(0,0)\} \\ 
\Sigma_\gamma^+ &\equiv \{z=0, y\geq 0\} \\ 
\Sigma_\gamma^- &\equiv \{z=0, y\leq 0\} \\ 
R_\alpha|_{\dot\gamma} &\equiv (1,0,0)|_{\R/\Z \times \{(0,0)\}} \\ 
\xi|_{\dot\gamma} &\equiv \{0\} \times \R^2|_{\R/\Z \times \{(0,0)\}} \\ 
i_{\V_{\dot\gamma}}(\dot \gamma^\perp) &\equiv \partial_y|_{\R/\Z \times \{(0,0)\}} \\ 
i_{\h_{\dot\gamma}}(\dot \gamma^\perp) &\equiv \partial_z|_{\R/\Z \times \{(0,0)\}}. 
\end{aligned}
\end{equation}

Denote by $X=(X_1,X_2,X_3)$ the Reeb vector field $R_\alpha$ in these coordinates and by $\psi_t$ its flow. Then $X(x,0,0)=(1,0,0)$ and since $\psi_t$ preserves the contact structure, we have 
\[
D\psi_t(x,0,0)\bigl[ \{0\} \times \R^2 \bigr] = \{0\} \times \R^2. 
\] 
A linearised solution $\zeta(t) = a_1(t) \partial_y + a_2(t) \partial_z$ along $\psi_t(x,0,0)=(x+t,0,0)$ satisfies 
\[
\left(\begin{array}{c} a_1'(t) \\ a_2'(t) \end{array}\right) = \left(\begin{array}{cc} 0 & -K(t) \\ 1 & 0 \end{array} \right) \left(\begin{array}{c} a_1(t) \\ a_2(t) \end{array} \right),
\]
where $K(t)$ is the Gaussian curvature at $\gamma(x+t)$. Writing in complex polar coordinates $a_1(t) + i a_2(t) = \rho(t)e^{i \theta(t)}$, for smooth functions $\rho\geq 0$ and $\theta$, we can easily check that 
\[
\theta'(t) = \cos^2 \theta (t) + K(t) \sin^2 \theta(t),\qquad \forall t\in \R.
\]
Therefore, the positivity of the Gaussian curvature along $\gamma$ implies the twist condition.  We have finished checking that $X$ meets all the assumptions of Lemma~\ref{lemma_return_time}. Proposition~\ref{prop_smoothness_return_time} follows readily from an application of that lemma.
\end{proof}

\subsection{The contact volume, the return time and the Riemannian area}\label{sec_volume_area}

As we have seen in the previous section, the Hilbert form $\lambda_H$ defined in~\eqref{hilb} induces by restriction a contact form $\alpha$ on $T^1 S^2$.  A further restriction produces the one-form $\lambda$ on the Birkhoff annulus $\Sigma_{\gamma}^+$. By using the standard smooth coordinates $(x,y) \in \R/L\Z\times [0,\pi]$ on $\Sigma_{\gamma}^+$, we express a vector $v\in \Sigma_{\gamma}^+$ as
\begin{equation}
\label{coordinate}
v= \cos y\ \dot\gamma(x) + \sin y\ \dot\gamma(x)^{\perp},
\end{equation}
and we find, using \eqref{hilb} and \eqref{vettori-tangenti}, together with (\ref{horiz}),
\[
\begin{split}
\lambda(v)[\partial_x] &= g_{\pi(v)} \bigl( v, d\pi(v)[\cos y\ i_{\h_v}(v) - \sin y\ i_{\h_v}(v^\perp)]\bigr) \\ &=
g_{\pi(v)}(v,\cos y \ v-\sin y \ v^\perp)= \cos y, \\
\lambda(v)[\partial_y]  &= g_{\pi(v)} \bigl( v, d\pi(v)[i_{\V_v}(v^\perp)]\bigr) = g_{\pi(v)}(v,0)=0 \\ 
\end{split}
\] 
Therefore, the expression of $\lambda$ in the coordinates $(x,y)$ is
\[
\lambda =\cos y \, dx,
\]
and its differential reads
\[
d\lambda = \sin y \, dx\wedge dy.
\]
Thus, the forms $\lambda$ and $\omega=d\lambda$ are the ones considered in part \ref{sez1} on the universal cover $S$ of $\R/L\Z\times [0,\pi]$.

Since the geodesic flow $\phi_t$ preserves $\alpha$ for all $t$, we have for any $v$ in $\mathrm{int}(\Sigma^+_\gamma)$ and $\zeta$ in $T_v\Sigma^+_\gamma$
\[
\begin{split}
(\varphi^* \lambda)(v)[\zeta] &= \lambda(\varphi(v))\bigl[d\varphi(v)[\zeta]\bigr] \\ &= \lambda\bigl(\phi_{\tau(v)}(v) \bigr)\bigl[ d\phi_{\tau(v)}(v)[\zeta] + d\tau(v)[\zeta] R_{\alpha}(\phi_{\tau(v)}(v)) \bigr] \\ &= \lambda(v)[\zeta] + d\tau(v)[\zeta]
\end{split}
\]
on $\mathrm{int}(\Sigma_{\gamma}^+)$, and hence on its closure $\Sigma_{\gamma}^+$ since all the objects here are smooth. Here, $R_{\alpha}$ is the Reeb vector field on the contact manifold $(T^1 S^2,\alpha)$, which coincides with the generator of the geodesic flow.
Therefore,
\[
d\tau = \varphi^* \lambda - \lambda \qquad \mbox{on } \Sigma_{\gamma}^+.
\] 
Now let 
\[
\Psi:\mathrm{int}( \Sigma_\gamma^+) \times \R \to T^1S^2 \setminus \bigl(\dot\gamma(\R) \cup- (\dot\gamma(\R)\bigr)
\]
be defined as $\Psi(v,t):=\phi_t(v)$. Then 
\[
\begin{split}
\Psi^* \alpha (v,t)[(\zeta,s)] &= \alpha(\phi_t(v)) \bigl[ d\phi_t(v)[\zeta] + s R_{\alpha}(\phi_t(v))\bigr] \\
&= \alpha(v)[\zeta] + s = \lambda(v)[\zeta]+s,
\end{split}
\]
that is,
\[
\Psi^* \alpha = \lambda + dt.
\]
Again, we used the preservation of $\alpha$ by $\phi_t$. Since $\lambda \wedge d\lambda=0$, being a three-form on a two-dimensional manifold, we deduce that
\[
\Psi^*(\alpha \wedge d\alpha) = dt \wedge d\lambda.
\]
Denoting by $K$ the subset 
\[
K:= \{(v,t)\in \mathrm{int}( \Sigma_\gamma^+ ) \times \R \mid v\in \mathrm{int}(\Sigma_\gamma^+), \ t\in [0, \tau(x)]\},
\]
we can relate the contact volume ${\rm Vol}(T^1S^2,\alpha)$ with the function $\tau$ as follows 
\[
\begin{split} {\rm Vol}(T^1S^2,\alpha) &=  \iiint_{T^1S^2\setminus \bigl(\dot\gamma(\R) \cup (-\dot\gamma(\R))\bigr)} \alpha\wedge d \alpha =\iiint_K \Psi^*(\alpha\wedge d\alpha) \\ &=\iiint_K dt \wedge d\lambda =  \iint_{\Sigma_\gamma^+} \left( \int_0 ^{\tau(v)} dt \right) \, d\lambda(v) = \iint_{\Sigma_\gamma^+} \tau \, d\lambda. 
\end{split}
\]
Summarizing, we have proved the following:

\begin{prop}\label{prop_map_lambda_volume}
The restriction $\lambda$ of the contact form $\alpha$ of $T^1 S^2$ to $\Sigma_{\gamma}^+$ has the form
\[
\lambda = \cos y\, dx
\]
in the standard coordinates $(x,y)\in \R/L\Z \times [0,\pi]$. The first return map $\varphi: \Sigma_{\gamma}^+ \rightarrow \Sigma_{\gamma}^+$ preserves $d\lambda$. Moreover, the first return time $\tau:\Sigma_{\gamma}^+ \rightarrow \R$ satisfies
\[
d\tau = \varphi^* \lambda - \lambda  \qquad \mbox{on } \Sigma_{\gamma}^+.
\]
Finally
\[
{\rm Vol}(T^1S^2,\alpha) = \iint_{\Sigma_\gamma^+} \tau \ d\lambda.
\]
\end{prop}

For completeness we state and prove below a well known fact.

\begin{prop}\label{prop_volume_area}
The contact volume of $(T^1 S^2,\alpha)$ and the Riemannian area of $(S^2,g)$ are related by the identity
\[
{\rm Vol}(T^1S^2,\alpha) = 2\pi \, {\rm Area}(S^2,g). 
\]
\end{prop}

\begin{proof}
Take isothermal coordinates $(x,y)\in U \subset \R^2$  on an embedded closed disk $U' \subset S^2$. In these coordinates, the metric $g$ takes the form 
\[
ds^2 = a(x,y)^2(dx^2 + dy^2),
\]
for a smooth positive function $a$. Any unit tangent vector $v \in T^1U' \subset T^1S^2$ can be written as
\[
v = \frac{\cos \theta}{a} \partial_x + \frac{\sin \theta}{a} \partial_y, \qquad \mbox{with} \qquad \theta\in \R / 2\pi \Z,
\] 
where $a=|\partial_x|_g=|\partial_y|_g$. Thus $(x,y,\theta)\in U \times \R/2\pi\Z$ can be taken as coordinates on $T^1 U'$, and the bundle projection becomes $\pi(x,y,\theta)=(x,y)$. With respect to these coordinates, the contact form
\[
\alpha(v)[\zeta] = g_{\pi(v)} \bigl(v, d\pi(v)[\zeta] \bigr)
\]
has the expression
\[
\alpha = a (\cos \theta \,dx + \sin \theta \,dy ).
\]
Differentiation yields
\[
d\alpha = da\wedge (\cos \theta \,dx + \sin \theta \, dy) + a(-\sin \theta \, d\theta \wedge dx+\cos \theta \, d\theta \wedge dy).
\]
Hence
\[
\begin{split}
\alpha\wedge d\alpha &= a\, da\wedge (\cos \theta \sin \theta \, dx\wedge dy+\sin \theta\cos \theta \, dy\wedge dx) \\
&\quad + a^2(\cos^2\theta \, dx\wedge d\theta\wedge dy-\sin^2 \theta \, dy\wedge d\theta\wedge dx) \\
&= a^2 \, dx\wedge d\theta\wedge dy = - a^2 \, dx \wedge dy \wedge d\theta.
\end{split}
\]
Therefore, the orientation of $T^1 U'$ which is induced by $\alpha \wedge d\alpha$ is opposite to the standard orientation of $U\times \R/2\pi \Z$, and we get
\[
\begin{split} 
{\rm Vol}(T^1U',\alpha)  &= \iiint_{T^1U'} \alpha \wedge d\alpha = \iiint_{U \times \R/2\pi \Z} a^2 dx\wedge dy \wedge d\theta \\ &= \iint_{U} a^2(x,y) \left( \int_0^{2\pi} d\theta \right) dxdy = 2\pi \iint_U a^2(x,y) \ dxdy \\ &= 2\pi \iint_U \sqrt{\det(g)} \ dxdy = 2\pi \, {\rm Area}(U',g).
\end{split}
\]
Taking two embedded disks $U',U''\subset S^2$ with disjoint interiors and coinciding boundaries, we get
\[
\begin{aligned} 
{\rm Vol}(T^1S^2,\alpha) &= {\rm Vol}(T^1U',\alpha)+{\rm Vol}(T^1U'',\alpha) \\ 
&= 2\pi({\rm Area}(U',g) + {\rm Area}(U'',g)) \\ 
&= 2\pi \,{\rm Area}(S^2,g). 
\end{aligned} 
\]
\end{proof}

\subsection{The flux and the Calabi invariant of the Birkhoff return map}

By using the standard smooth coordinates $(x,y)$ given by (\ref{coordinate}), we can identify the Birkhoff annulus $\Sigma_{\gamma}^+$ with $\R/L \Z \times [0,\pi]$. Its universal cover is the natural projection
\[
p : S \rightarrow \Sigma_{\gamma}^+,
\]
where $S$ is the strip $\R \times [0,\pi]$. The first return map $\varphi: \Sigma_{\gamma}^+ \rightarrow \Sigma_{\gamma}^+$ preserves the two-form $\omega=d\lambda$ and maps each boundary component into itself. Therefore, $\varphi$ can be lifted to a diffeomorphism in the group $\mathcal{D}_L(S,\omega)$ which is considered in part \ref{sez1}. The aim of this section is to prove the following result, which relates the objects of this part with those of part \ref{sez1}.

\begin{thm}\label{relation-thm}
Assume that the metric $g$ on $S^2$ is $\delta$-pinched with $\delta> 1/4$. Let $\gamma$ be a simple closed geodesic of length $L$ on $(S^2,g)$. Then the first return map $\varphi: \Sigma_{\gamma}^+ \rightarrow \Sigma_{\gamma}^+$ has a lift $\Phi: S \rightarrow S$ which belongs to $\mathcal{D}_L(S,\omega)$ and has the following properties: 
\begin{enumerate}[(i)]
\item $\Phi$ has zero flux.
\item The first return time $\tau:\Sigma_{\gamma}^+ \rightarrow \R$ is related to the action $\sigma: S \rightarrow \R$ of $\Phi$ by the identity
\[
\tau\circ p =  L+ \sigma \qquad \mbox{on } S.
\]
\item The area of $(S^2,g)$ is related to the Calabi invariant of $\Phi$ by the identity
\[
\pi \, \mathrm{Area}(S^2,g) = L^2 + L\ \mathrm{CAL}(\Phi).
\]
\end{enumerate}
\end{thm}

The proof of this theorem requires an auxiliary lemma, which will play an important role also in the next section.

\begin{lem}\label{lema_naointersecta}
Assume that $(S^2,g)$ is $\delta$-pinched for some $\delta>1/4$. Fix some $v$ in $\Sigma_{\gamma}^{\pm}$ and denote by $\alpha$ the geodesic satisfying $\dot\alpha(0)=v$. Then the geodesic arc $\alpha|_{[0,\tau^{\pm}(v)]}$ is injective.
\end{lem}

\begin{proof}
We consider the case of $\Sigma_{\gamma}^+$, the case of $\Sigma_{\gamma}^-$ being completely analogous. Up to the multiplication of $g$ by a positive number, we may assume that $1\leq K < 4$.

Let $x^*\in \R$ be such that $\alpha(0) = \gamma(x^*)$ and let $y^*\in [0,\pi]$ be the angle between $\dot \gamma(x^*)$ and $v=\dot \alpha(0)$. Consider the family of unit speed geodesics $\alpha_y$ with $\alpha_y(0)=\alpha(0)=\gamma(x^*)$ such that the angle from $\dot \gamma(x^*)$ to $v_y := \dot \alpha_y(0)$  is $y$, for $y \in [0,\pi]$. In particular, $\alpha_{y^*}=\alpha$ and $v_{y^*}=v$. By Proposition \ref{prop_smoothness_return_time} (i), 
\[
\{\alpha_{y}|_{[0,\tau_+(v_y)]}\}_{y\in [0,\pi]}
\] 
is a smooth family of geodesic arcs, parametrised on a family of intervals whose length varies smoothly.

We claim that $\tau_+(v_0) <L$ and $\tau_+(v_{\pi})<L$. In order to prove this, first notice that the length $L$ of the closed geodesic $\gamma$ satisfies
\begin{equation}\label{est-L}
L \geq \frac{2\pi}{\sqrt{\max K}} > \frac{2\pi}{\sqrt{4}} = \pi,
\end{equation}
thanks to the lower bound (\ref{injradius-estimate}) on the injectivity radius and to the inequality $K<4$. 
Moreover, by Proposition \ref{prop_smoothness_return_time} (i) the number $\tau_+(v_0)$ is the first positive zero of the solution $u$ of the Jacobi equation
\[
u''(t) + K(\gamma(x^*+t)) u(t) = 0, \qquad u(0)=0, \qquad u'(0)=1.
\]
Writing the complex function $u'+iu$ in polar coordinates as $u'+iu=r e^{i\theta}$, for smooth real functions $r>0$ and $\theta$ satisfying $r(0)=1$, $\theta(0)=0$, a standard computation gives
\[
\theta'(t) = \cos^2 \theta(t) + K(\gamma(x^*+t)) \sin^2 \theta(t).
\]
Since $K\geq 1$, we have $\theta'\geq 1$ and hence $\theta(L)\geq L > \pi$. This implies that $\tau_+(v_0)<L$. The case of $\tau_+(v_{\pi})$ follows by applying the previous case to the geodesic $t\mapsto \gamma(-t)$. 

Let $Y_0$ be the subset of $[0,\pi]$ consisting of those $y$ for which $\alpha_{y}|_{[0,\tau_+(v_y)]}$ is injective.  The set $Y_0$ is open in $[0,\pi]$, and by the above claim 0 and $\pi$ belong to $Y_0$. Let $Y_1$ be the subset of $(0,\pi)$ consisting of those $y$ for which $\alpha_{y}|_{[0,\tau_+(v_y)]}$ has an interior self-intersection: There exist $0<s<t<\tau_+(v_y)$ such that $\alpha_{y}(s)=\alpha_{y}(t)$. Such an interior self-intersection must be transverse, so the fact that $S^2$ is two-dimensional implies that also $Y_1$ is open in $[0,\pi]$. It is enough to show that $Y_0 \cup Y_1 = [0,\pi]$: Indeed, if this is so, the fact that $[0,\pi]$ is connected implies that only one of the two open sets $Y_0$ and $Y_1$ can be non-empty, and we have already checked that $Y_0$ contains $0$ and $\pi$. The conclusion is that $[0,\pi]=Y_0$, and in particular $\alpha=\alpha_{y^*}$ is injective.

If $y$ belongs to the complement of $Y_0 \cup Y_1$ in $[0,\pi]$, then $y\in (0,\pi)$ and $\alpha_{y}|_{[0,\tau_+(v_y)]}$ has a self-intersection only at its endpoints: $\alpha|_{[0,\tau_+(v_y))}$ is injective and $\alpha_{y}(\tau_+(v_y)) = \alpha_{y}(0)$.
Denote by $l>0$ the length of the geodesic loop $\alpha_{y}|_{[0,\tau_+(v_y)]}$. 
Together with the closed curve $\gamma$, this geodesic loop forms a two-gon with perimeter equal to $L+l$. By Theorem~\ref{thm_est_conv_polygon} and the inequality $K\geq 1$, its perimeter $L+l$ satisfies 
\[
L+l \leq \frac{2\pi}{\sqrt{\min K}}\leq 2\pi.
\] 
By using the bound (\ref{est-L}) and the analogous bound $l>\pi$ for the geodesic loop $\alpha_{y}|_{[0,\tau_+(v_y)]}$, we obtain
\[
L+l > 2\pi. 
\]
The above two estimates contradict each other, and this shows that the complement of $Y_0 \cup Y_1$ is empty, concluding the proof.
\end{proof}

\begin{proof}[Proof of Theorem \ref{relation-thm}]
Given $v\in T^1 S^2$, we denote by $\alpha_v$ the geodesic parametrised by arc length such that $\dot\alpha_v(0)=v$. Let $v\in \Sigma_\gamma^+$ with $\pi(v)=\gamma(x)$.
Then we know from Lemma \ref{lema_naointersecta} that the geodesic arc $\alpha_v|_{[0,\tau_+(v)]}$ is injective. In particular, $\alpha_v(\tau_+(v))$ is distinct from $\alpha_v(0)=\gamma(x)$, so there exists a unique number
\[
\rho_+(v)\in (0,L)
\]
such that 
\[
\alpha_v(\tau_+(v)) = \gamma(x+\rho_+(v)).
\]
By the continuity of the geodesic flow and of the function $\tau_+$, the function
\[
\rho_+ : \Sigma_{\gamma}^+ \rightarrow (0,L)
\]
is continuous.  The restriction of $\tau_+$ to the boundary of $\Sigma_{\gamma}^+$ satisfies
\begin{equation}
\label{bordo1}
\rho_+(\dot\gamma(x)) = \tau_+(\dot\gamma(x)) \qquad \mbox{and} \qquad  \rho_+(-\dot\gamma(x)) = L -\tau_+(-\dot\gamma(x)), \qquad \forall x\in \R.
\end{equation}
Similarly, there exists a unique continuous function 
\[
\rho_-: \Sigma_{\gamma}^- \rightarrow (0,L)
\]
such that, if $v\in \Sigma_{\gamma}^-$ is based at $\gamma(x)$, we have
\[
\alpha_v(\tau_-(v)) = \gamma(x+\rho_-(v)).
\]
As before,
\begin{equation}
\label{bordo2}
\rho_-(\dot\gamma(x)) = \tau_-(\dot\gamma(x)) \qquad \mbox{and} \qquad  \rho_-(-\dot\gamma(x)) = L -\tau_-(-\dot\gamma(x)), \qquad \forall x\in \R.
\end{equation}
Define the function
\[
\rho: \Sigma_{\gamma}^+ \rightarrow (0,2L)
\]
by
\[
\rho := \rho_+ + \rho_- \circ \varphi_+.
\]
By construction, we have for every $v\in \Sigma_{\gamma}^+$ with $\pi(v)=\gamma(x)$, 
\begin{equation}
\label{larho}
\pi(\varphi(v)) = \gamma(x+\rho(v)),
\end{equation}
and, by (\ref{bordo1}) and (\ref{bordo2}), together with (\ref{comp-tau}),
\begin{equation}
\label{bordo3}
\rho(\dot\gamma(x)) = \tau(\dot\gamma(x)) \qquad \mbox{and} \qquad  \rho(-\dot\gamma(x)) = 2L -\tau(-\dot\gamma(x)), \qquad \forall x\in \R.
\end{equation}
Using the standard coordinates $(x,y)\in \R/L\Z \times [0,\pi]$ on $\Sigma_\gamma^+$, we can see $\rho$ and $\tau$ as functions on $\R/L\Z \times [0,\pi]$ or, equivalently, as functions on $\R \times [0,\pi]$ which are $L$-periodic in the first variable. Thanks to (\ref{larho}) we can fix a lift $\Phi=(X,Y)\in \mathcal{D}_L(S,\omega)$ of $\varphi$ by requiring its first component to be given by
\begin{equation}
\label{scelta-lift}
X(x,y) =  x + \rho(x,y) - L.
\end{equation}
By (\ref{bordo3}) we have
\begin{equation}
\label{Delta-al-bordo}
X(x,0) - x= \tau(x,0)-L, \qquad  X(x,\pi) - x = L - \tau(x,\pi), \qquad \forall x\in \R.
\end{equation}
By definition, the action $\sigma: S \rightarrow \R$ of $\Phi$ is uniquely determined by the conditions
\[
\begin{split}
d\sigma &= \Phi^* \lambda - \lambda, \\
\sigma(x,0) + \mathrm{FLUX}(\Phi) &= \int_{\gamma_x} \lambda = X(x,0)-x , \qquad \forall x\in \R.
\end{split}
\]
where $\gamma_x$ is a path in $\partial S$ connecting $(x,0)$ to $\Phi(x,0)=(X(x,0),0)$. By the first identity in (\ref{Delta-al-bordo}) we have
\[
\sigma(x,0) + \mathrm{FLUX}(\Phi) = \tau(x,0) - L, \qquad \forall x\in \R.
\]
By Proposition \ref{prop_map_lambda_volume}, also the $(L,0)$-periodic function $\tau: S \rightarrow \R$ satisfies $d\tau = \Phi^* \lambda - \lambda$, so the above identity implies that
\begin{equation}
\label{sigma-tau}
\sigma(x,y) + \mathrm{FLUX}(\Phi) = \tau(x,y) - L, \qquad \forall (x,y)\in S.
\end{equation}
By Proposition \ref{salto} and the second identity in (\ref{Delta-al-bordo}) we have
\[
\sigma(x,\pi) - \mathrm{FLUX}(\Phi) = \int_{\delta_x} \lambda = - X(x,\pi) + x = \tau(x,\pi) - L, \qquad \forall x\in \R,
\]
where $\delta_x$ is a path in $\partial S$ connecting $(x,\pi)$ to $\Phi(x,\pi)=(X(x,\pi),\pi)$. Together with (\ref{sigma-tau}) this implies that $\mathrm{FLUX}(\Phi)=0$, thus proving statement (i).  Statement (ii) now follows from (\ref{sigma-tau}). 

By Propositions \ref{prop_volume_area} and \ref{prop_map_lambda_volume}, we have
\[
\begin{split}
\pi \, \mathrm{Area} (S^2,g) &= \frac{1}{2} \, \mathrm{Vol} (T^1S^2,\alpha) = \frac{1}{2} \iint_{\R/L\Z \times [0,\pi]} \tau\, d\lambda = \frac{1}{2} \iint_{[0,L] \times [0,\pi]} (L+\sigma)\, d\lambda \\ &= L^2 + \frac{1}{2} \iint_{[0,L] \times [0,\pi]} \sigma\, d\lambda = L^2 + L \ \mathrm{CAL}(\Phi),
\end{split}
\]
and (iii) is proved.
\end{proof}

\subsection{Proof of the monotonicity property}\label{sec_existence_generating}

As we have seen, the first return map $\varphi$ can be lifted to a diffeomorphism $\Phi$ in the class $\mathcal{D}_L(S,\omega)$. The aim of this section is to prove that, if the curvature is sufficiently pinched, then this lift is a monotone map, in the sense of Definition \ref{mon-defn} (notice that the monotonicity does not depend on the choice of the lift).

\begin{prop}\label{prop_pinched_partial}
If $g$ is $\delta$-pinched for some $\delta> (4+\sqrt{7})/8$, then any lift $\Phi: S \rightarrow S$ of the first return map $\varphi: \Sigma_{\gamma}^+ \rightarrow \Sigma_{\gamma}^+$ is monotone.
\end{prop}

\begin{proof}
We may assume that the values of the curvature lie in the interval $[\delta,1]$, where $\delta>(4+\sqrt{7})/8$. 

Fix some $x^*\in \R$. In order to simplify the notation in the next computations, we set for every $y\in [0,\pi]$
\[
l_y:=\tau(x^*,y), \qquad t_y := X(x^*,y), \qquad \tilde y(y):= Y(x^*,y),
\]
where $\tau$ is seen as a $(L,0)$-periodic function on $S$ and $X$ and $Y$ are the components of the fixed lift $\Phi=(X,Y)$ of $\varphi$. Our aim is to show that the derivative of the function $\tilde{y}$ is positive on $[0,\pi]$.

Consider the $1$-parameter geodesic variation
\[
\alpha_y(t) := \exp_{\gamma(x^*)}[t(\cos y\ \dot\gamma(x^*)+\sin y\ \dot\gamma(x^*)^\bot)],
\]
where $y \in [0,\pi]$. For each $y\in (0,\pi)$, $l_y$ is the second time $\alpha_y(t)$ hits $\gamma(\R)$ or, equivalently, the first time $\dot\alpha_y(t)$ hits $\Sigma^+_\gamma$. Moreover, $\alpha_0(t)=\gamma(x^*+t)$, and $l_0$ is the time to the second conjugate point to $\alpha_0(0)$ along $\alpha_0$; analogously, $\alpha_\pi(t)=\gamma(x^*-t)$, and $l_\pi$ is the time to the second conjugate point to $\alpha_{\pi}(0)$ along $\alpha_{\pi}$. By construction
\[
\alpha_y(l_y) = \gamma(t_y),
\]
and
\begin{equation}\label{eqalphad}
\begin{aligned}  
\dot \alpha_y(l_y) & = \cos \tilde y \ \dot \gamma(t_y) + \sin \tilde y \ \dot \gamma(t_y)^\perp, \\ 
\dot \alpha_y(l_y)^\perp &= - \sin \tilde y \ \dot \gamma(t_y) + \cos \tilde y \ \dot \gamma(t_y)^\perp,
\end{aligned}  
\end{equation}
for every $y\in [0,\pi]$, where the function $\tilde y$ is evaluated at $y$.
Since $\gamma$ is a geodesic,
\[
\frac{D}{dy} \dot\gamma\circ t_y = \frac{D}{dt} \dot{\gamma} (t_y) \frac{\partial t_y}{\partial y} = 0,
\]
and since the vector field $\dot\gamma^{\perp}$ along $\gamma$ is parallelly transported,
\[
\frac{D}{dy} \dot\gamma^{\perp} \circ t_y = \frac{D}{dt} \dot{\gamma} (t_y)^{\perp} \frac{\partial t_y}{\partial y} = 0.
\]
Notice that $V(y):=\dot \alpha_y(l_y)$ is a vector field along the smooth curve $y\mapsto \gamma(t_y)$.  
Using that $\gamma$ is a geodesic we obtain from \eqref{eqalphad} 
\begin{align} \frac{ DV}{dy}(y) &= -\tilde y' \sin \tilde y \ \dot\gamma(t_y) + \cos \tilde y  \ \frac{D}{dy} \dot\gamma\circ t_y +\tilde y'\cos \tilde y \ \dot \gamma(t_y)^\perp + \sin \tilde y \ \frac{D}{dy} \dot\gamma^{\perp} \circ t_y\nonumber \\ 
&= -\tilde y' \sin \tilde y \ \dot\gamma(t_y) +\tilde y'\cos \tilde y \ \dot \gamma(t_y)^\perp\nonumber \\ 
&= \tilde y'(y) \ \dot \alpha_y(l_y)^\perp.  \label{covd}
\end{align} 
The geodesic variation $\{\alpha_y\}$ at $y=y^*$ corresponds to the Jacobi field $J$ along $\alpha_{y^*}$ given by \begin{equation}\label{Jacobi-J}
J(t) := \left.\frac{\partial}{\partial y}\right|_{y=y^*}\alpha_y(t).
\end{equation}
From the initial conditions $J(0)=0$ and 
\[
\frac{DJ}{dt}(0) = \left. \frac{D}{dy}\right|_{y=y^*}\dot \alpha_{y}(0) = \left.\frac{d}{dy}\right|_{y=y^*} \dot \alpha_{y}(0) = \dot \alpha_{y^*}(0)^\perp,
 \]
we find a smooth real function $u$ such that 
\[
J(t)=u(t)\dot\alpha_{y^*}(t)^\perp, \qquad \frac{DJ}{dt}(t) = u'(t)\dot\alpha_{y^*}(t)^\perp, \qquad \forall t\in \R,
\] 
and
\begin{equation}\label{initial-cond}
u(0) =0 , \qquad u'(0)=1.
\end{equation}
Moreover
\begin{equation}\label{covd2} 
\left. \frac{D}{dy}\right|_{y=y^*} \dot \alpha_y(t) = \frac{D}{dt} J(t) = u'(t)\dot\alpha_{y^*}(t)^\perp, \qquad \forall t\in \R.
\end{equation}

Recall that the covariant derivative of a vector field $v$ along a curve $\delta$ on $S^2$ is the full derivative of the corresponding curve $(\delta,v)$ on $TS^2$ projected back to $TS^2$ by the connection operator $K:TTS^2 \to TS^2$. More precisely, $K$ projects this full derivative $(\delta,v)'$ onto the vertical subspace $\V_{(\delta,v)}\subset T_{(\delta,v)}TS^2$ along the horizontal subspace $\h_{(\delta,v)}\subset T_{(\delta,v)}TS^2$, and then brings it to $T_\delta S^2$ via the inverse of the isomorphism $i_{\V_{v}}$, see the discussion after the proof of Lemma~\ref{lemma_return_time}. 
In \eqref{covd} we find the covariant derivative of the vector field $y\mapsto \dot \alpha_y(l_y)$ along the curve $y\mapsto \alpha_y(l_y)$. In \eqref{covd2} we see the covariant derivative of the vector field $y \mapsto \dot \alpha_y(t)$ along the curve $y\mapsto \alpha_y(t)$ for fixed $t$.  Since $\alpha_y$ is a geodesic for all $y$, by using the above description of the covariant derivative we get from \eqref{covd} and \eqref{covd2}
\[
\begin{split}
\tilde y'(y^*) \dot{\alpha}_{y^*}(l_{y^*})^\perp & = \frac{DV}{dy}(y^*)  = \left.\frac{D}{dy}\right|_{y=y^*} \dot \alpha_y (l_{y^*}) + l_y'(y^*) \left.\frac{D}{dt}\right|_{t=l_{y^*}} \dot \alpha_{y^*}(t) \\
&= \left.\frac{D}{dy}\right|_{y=y^*} \dot \alpha_y (l_{y^*}) = u'(l_{y^*})\dot\alpha_{y^*}(l_{y^*})^\perp,
\end{split}
\]
for every $y^*\in [0,\pi]$, from which we derive the important identity
\begin{equation}\label{crucial_a'}
\tilde y'(y^*) = u'(l_{y^*}), \qquad \forall y^* \in [0,\pi].
\end{equation}

Write $$ l_{y^*} = l+l' $$ for $y^* \in (0,\pi)$, where $l>0$ is the first time $\alpha_{y^*}(t)$ hits $\gamma$, that is,
\[
l = \tau_+(\dot\alpha_{y^*}(0)), \qquad l' = \tau_-\bigl( \varphi_+(\dot\alpha_{y^*}(0))\bigr).
\]
By Lemma~\ref{lema_naointersecta}, $\alpha_{y^*}|_{[0,l]}$ is injective and, in particular, its end-points are distinct points of $\gamma$, dividing it into two segments $\gamma_1,\gamma_2$ with lengths $l_1,l_2>0$, respectively, and $l_1+l_2=L$.  
Therefore, $\alpha_{y^*}|_{[0,l]}$ and $\gamma_1$ determine a geodesic two-gon. The same holds with $\alpha_{y^*}|_{[0,l]}$ and $\gamma_2$.  It follows from Theorem~\ref{thm_est_conv_polygon} that
\[
l_1+l\leq \frac{2\pi}{\sqrt{\delta}} \qquad \mbox{and} \qquad l_2+l\leq \frac{2\pi}{\sqrt{\delta}}.
\]
Theorem~\ref{thm_est_conv_polygon} also implies that $L\leq 2\pi/\sqrt{\delta}$. From Klingenberg's lower bound \eqref{injradius-estimate} on the injectivity radius of $g$, we must have $l_1+l\geq 2\pi$, $l_2+l\geq 2\pi$, and $L\geq 2\pi$. Putting these inequalities together, we obtain 
\begin{eqnarray}
\label{eq_comprimento} & 2\pi  \leq l_i+l \leq \displaystyle{\frac{2\pi}{\sqrt{\delta}}}, \qquad i=1,2, \\ \label{eq_comprimentob} & 2\pi \leq L=l_1+l_2 \leq \displaystyle{\frac{2\pi}{\sqrt{\delta}}}. 
\end{eqnarray} 
By adding the inequalities~\eqref{eq_comprimento}, we obtain 
\begin{equation}
4\pi \leq 2l + L \leq \frac{4\pi}{\sqrt{\delta}}.  
\end{equation} 
Together with \eqref{eq_comprimentob}, the above inequality implies
\[
2\pi - \frac{\pi}{\sqrt{\delta}} \leq l \leq \frac{2\pi}{\sqrt{\delta}} - \pi. 
\]
Arguing analogously with the geodesic arc $\alpha_{y^*}|_{[l,l_{y^*}=l+l']}$, we obtain the similar estimate 
\[
2\pi - \frac{\pi}{\sqrt{\delta}} \leq l' \leq \frac{2\pi}{\sqrt{\delta}} - \pi,
\]
concluding  that the length $l_{y^*}$ of $\alpha_{y^*}$ satisfies  
\begin{equation} \label{eq_comprimento3}
4\pi - \frac{2\pi}{\sqrt{\delta}} \leq l_{y^*} = l+ l' \leq \frac{4\pi}{\sqrt{\delta}} - 2\pi. 
\end{equation}

The Jacobi equation for the vector field $J$ along $\alpha_{y^*}$ which is defined in (\ref{Jacobi-J}) can be written in terms of the scalar function $u$ as
\[
u''(t) + K(\alpha_{y^*}(t)) u(t) = 0.
\]
Writing 
\[
u(t)'+ i u(t) = r e^{i \theta}
\]
for smooth real functions $r>0$ and $\theta$, we get 
\begin{equation}\label{eq_etad2} 
\theta' = \cos^2\theta + K(\alpha_{y^*}) \sin^2 \theta.
\end{equation} 
The initial conditions (\ref{initial-cond}) imply that $r(0)=1$ and $\theta(0)=0$. From~\eqref{eq_etad2}
we have $\delta \leq \theta' \leq 1$. Hence, from the estimate for $l_{y^*}$ given in~\eqref{eq_comprimento3}, we find
\begin{equation}\label{deltaeta}
\delta\left(4\pi - \frac{2\pi}{\sqrt{\delta}} \right)\leq \theta(l_{y^*})\leq \frac{4\pi}{\sqrt{\delta}} - 2\pi.  
\end{equation}
From $\delta>(4+\sqrt 7)/8$ we get 
\[
\delta\left(4\pi - \frac{2\pi}{\sqrt{\delta}} \right) > \frac{3\pi}{2},
\]
and since a fortiori $\delta>64/81$, we have also
\[
\frac{4\pi}{\sqrt{\delta}} - 2\pi < \frac{5\pi}{2}.
\]
Therefore, (\ref{deltaeta}) implies that $\cos \theta(l_{y^*})$ is positive. By the identity~\eqref{crucial_a'}, we conclude that
\[
\tilde y'(y^*) = u'(l_{y^*}) = r(l_{y^*}) \cos \theta(l_{y^*})>0,
\]
as we wished to prove.
\end{proof}

\subsection{Proof of the main theorem}

In \cite{cc92} Calabi and Cao have proved that any shortest closed geodesic on a two-sphere with non-negative curvature is simple. If one assumes that the curvature is suitably pinched, this fact follows also from the lower bound \eqref{injradius-estimate} on the injectivity radius and from Theorem~\ref{thm_est_conv_polygon}:

\begin{lem}
\label{simple}
Assume that the metric $g$ on $S^2$ is $\delta$-pinched for some $\delta>1/4$. Then
any closed geodesic $\gamma$ of minimal length on $(S^2,g)$ is a simple curve.
\end{lem}

\begin{proof}
If a closed geodesic $\gamma$ of minimal length is not simple, then it contains at least two distinct geodesic loops. By the lower bound \eqref{injradius-estimate} on the injectivity radius, each of these two geodesic loops has length at least
\[
\frac{2\pi}{\sqrt{\max K}},
\] 
and we deduce that
\begin{equation}\label{lunga}
L \geq \frac{4\pi}{\sqrt{\max K}}.
\end{equation}
A celebrated theorem due to Lusternik and Schnirelmann implies the existence of simple closed geodesics on any Riemannian $S^2$. By Theorem~\ref{thm_est_conv_polygon} any simple closed geodesic has length at most
\[
\frac{2\pi}{\sqrt{\min K}}.
\]
By the pinching assumption, 
\[
\frac{2\pi}{\sqrt{\min K}} \leq \frac{2\pi }{\sqrt{\delta \max K}} < \frac{4\pi}{\sqrt{\max K}},
\]
so by \eqref{lunga} any simple closed geodesic is shorter than $L$. This contradicts the fact that $L$ is the minimal length of a closed geodesic and proves that $\gamma$ must be simple.  
\end{proof}

Now let $\gamma$ be a simple closed geodesic on $(S^2,g)$ of length $L$.
Let $\varphi:\Sigma_{\gamma}^+ \rightarrow \Sigma_{\gamma}^+$ be the associated Birkhoff first return map and let $\Phi\in \mathcal{D}_L(S,\omega)$ be the lift of $\varphi$ with zero flux whose existence is guaranteed by Theorem \ref{relation-thm}. Here is a first consequence of Theorem \ref{relation-thm}:

\begin{lem}\label{zoll-lem}
Assume that the metric $g$ on $S^2$ is $\delta$-pinched for some $\delta>1/4$. Then $g$ is Zoll if and only if $\Phi=\mathrm{id}$.
\end{lem}

\begin{proof}
Assume that $\Phi=\mathrm{id}$. Then the action $\sigma$ of $\Phi$  is identically zero, so by Theorem \ref{relation-thm} (ii) the first return time function $\tau$ is identically equal to $L$. Therefore, all the vectors in the interior of $\Sigma_{\gamma}^+$ are initial velocities of closed geodesics of length $L$. Since also the vectors in the boundary of $\Sigma_{\gamma}^+$ are by construction initial velocities of  closed geodesics of length $L$, we deduce that all the geodesics on $(S^2,g)$ are closed and have length $L$.

Conversely assume that $(S^2,g)$ is Zoll. Since $\gamma$ has length $L$, all the geodesics on $(S^2,g)$ are closed and have length $L$. 
Then every $v$ in $\mathrm{int}(\Sigma_{\gamma}^+)$ is a periodic point of $\varphi$, i.e.\ there is a minimal natural number $k(v)$ such that 
$\varphi^{k(v)}(v)=v$, and the identity 
\[
\sum_{j=0}^{k(v)-1} \tau(\varphi^j(v)) = L
\]
holds on $\mathrm{int}(\Sigma_{\gamma}^+)$. Thanks to the continuity of $\tau$ and $\varphi$ and to the positivity of $\tau$, the above identity forces the function $k$ to be constant, $k\equiv k_0\in \N$. By continuity, the above identity holds also on the boundary of $\Sigma_{\gamma}^+$, and we have in particular
\[
\sum_{j=0}^{k_0-1} \tau(\varphi^j(\dot{\gamma}(t))) = L \qquad \forall t\in \R/L \Z.
\]
By the above identity, there exists $t_0\in \R/L \Z$ such that
\[
\tau(\dot{\gamma}(t_0)) \leq \frac{L}{k_0},
\]
that is, the time to the second conjugate point to $\gamma(t_0)$ along $\gamma$ is at most $L/k_0$. Since this time is at least twice the injectivity radius of $(S^2,g)$, we obtain from (\ref{injradius-estimate})
\begin{equation}\label{in1}
\frac{L}{k_0} \geq \tau(\dot{\gamma}(t_0)) \geq 2 \, \mathrm{inj}(g) \geq \frac{2\pi}{\sqrt{\max K}}. 
\end{equation}
On the other hand, by Theorem \ref{thm_est_conv_polygon} and by the pinching assumption, the length $L$ of the simple closed geodesic $\gamma$ satisfies
\begin{equation}\label{in2}
L \leq \frac{2\pi}{\sqrt{\min K}} \leq \frac{2\pi}{\sqrt{\delta \max K}} < \frac{4\pi}{\sqrt{\max K}}.
\end{equation}
Inequalities (\ref{in1}) and (\ref{in2}) imply that the positive integer $k_0$ is less than 2, hence $k_0=1$ 
and $\varphi=\mathrm{id}$.
Then $\Phi$ is a translation by an integer multiple of $L$ and, having zero flux, it must be the identity.
\end{proof}

The theorem which is stated in the introduction concerns two inequalities, which we treat separately in the following two statements.

\begin{thm}
\label{main1}
If $g$ is $\delta$-pinched with $\delta> (4+\sqrt{7})/8$, then 
\begin{equation}\label{the-bound}
\ell_{\min}(g)^2 \leq \pi \, \mathrm{Area}(S^2,g), 
\end{equation}
and the equality holds if and only if $(S^2,g)$ is Zoll.
\end{thm}

\begin{proof}
Let $\gamma$ be a shortest closed geodesic on $(S^2,g)$ and let $L$ be its length. Since in particular $\delta>1/4$, Lemma \ref{simple} implies that $\gamma$ is simple. Let $\Phi\in \mathcal{D}_L(S,\omega)$ be the lift with zero flux of the Birkhoff first return map which is associated to $\gamma$.

If $(S^2,g)$ is Zoll, then by the Lemma \ref{zoll-lem} $\Phi=\mathrm{id}$, so $\mathrm{CAL}(\Phi)=0$, and Theorem \ref{relation-thm} (iii) implies that
\[
\pi\, \mathrm{Area} (S^2,g) = L^2.
\]
This shows that if $g$ is Zoll, then the equality holds in (\ref{the-bound}).

There remains to show that if $(S^2,g)$ is not Zoll, then the strict inequality holds in (\ref{the-bound}). Assume by contradiction that
\[
L^2 \geq \pi \, \mathrm{Area}(S^2,g).
\]
Then by Theorem \ref{relation-thm} (iii) we have
\[
L\ \mathrm{CAL}(\Phi) = \pi \, \mathrm{Area}(S^2,g) - L^2 \leq 0,
\]
and $\mathrm{CAL}(\Phi)$ is non-positive.
Since $(S^2,g)$ is not Zoll, by Lemma \ref{zoll-lem} the map $\Phi$ is not the identity. By Proposition \ref{prop_pinched_partial}, $\Phi$ satisfies the hypothesis of Theorem \ref{thm_calabi}, which guarantees the existence of a fixed point $(x,y)\in \mathrm{int}(S)$ of $\Phi$ with action $\sigma(x,y)<0$. The geodesic which is determined by the corresponding vector in $\Sigma_{\gamma}^+$ is closed and, by Theorem \ref{relation-thm} (ii), has length
\[
\tau(x,y) = L + \sigma(x,y) < L.
\]
This contradicts the fact that $L$ is the minimal length of a closed geodesic. This contradiction implies that when $(S^2,g)$ is not Zoll, then the strict inequality 
\[
L^2 < \pi \, \mathrm{Area}(S^2,g)
\]
holds.
\end{proof}

The proof of the second inequality differs only in a few details:

\begin{thm}
\label{main2}
If $g$ is $\delta$-pinched with $\delta> (4+\sqrt{7})/8$, then 
\begin{equation}\label{the-bound2}
\ell_{\max}(g)^2 \geq \pi \, \mathrm{Area}(S^2,g), 
\end{equation}
and the equality holds if and only if $(S^2,g)$ is Zoll.
\end{thm}

\begin{proof}
Let $\gamma$ be a longest simple closed geodesic on $(S^2,g)$ and let $L$ be its length. Let $\Phi\in \mathcal{D}_L(S,\omega)$ be the lift with zero flux of the Birkhoff first return map which is associated to $\gamma$.

If $(S^2,g)$ is Zoll, then by the Lemma \ref{zoll-lem} $\Phi=\mathrm{id}$, so $\mathrm{CAL}(\Phi)=0$, and Theorem \ref{relation-thm} (iii) implies that
\[
\pi\, \mathrm{Area} (S^2,g) = L^2.
\]
This shows that if $g$ is Zoll, then the equality holds in (\ref{the-bound2}).

There remains to show that if $(S^2,g)$ is not Zoll, then the strict inequality holds in (\ref{the-bound2}). Assume by contradiction that
\[
L^2 \leq \pi \, \mathrm{Area}(S^2,g).
\]
Then by Theorem \ref{relation-thm} (iii) we have
\[
L\ \mathrm{CAL}(\Phi) = \pi \, \mathrm{Area}(S^2,g) - L^2 \geq 0,
\]
and $\mathrm{CAL}(\Phi)$ is non-negative.
Since $(S^2,g)$ is not Zoll, by Lemma \ref{zoll-lem} the map $\Phi$ is not the identity. By Proposition \ref{prop_pinched_partial}, $\Phi$ satisfies the hypothesis of Theorem \ref{thm_calabi}, which guarantees the existence of a fixed point $(x,y)\in \mathrm{int}(S)$ of $\Phi$ with action $\sigma(x,y)>0$. The geodesic which is determined by the corresponding vector in $\Sigma_{\gamma}^+$ is closed and, by Theorem \ref{relation-thm} (ii), has length
\[
\tau(x,y) = L + \sigma(x,y) > L.
\]
Moreover, Lemma \ref{lema_naointersecta} implies that this closed geodesic is simple. This contradicts the fact that the longest simple closed geodesic has length $L$ and proves that the strict inequality 
\[
L^2 > \pi \, \mathrm{Area}(S^2,g)
\]
holds. The proof is complete.
\end{proof}

\begin{rem}
The proof of our main theorem uses the bound $\delta> (\sqrt{7}+4)/8$ on the pinching constant $\delta$ only to have the monotonicity of the map $\Phi$. If the fixed point theorem \ref{thm_calabi} holds without this assumption, then the conclusion of our main theorem holds under the weaker condition $\delta>1/4$.
\end{rem}

\appendix

\section{Toponogov's theorem and its consequences}

This appendix is devoted to explaining how to estimate lengths of convex geodesic polygons using a relative version of Toponogov's theorem.

\subsection{Geodesic polygons and their properties}\label{sec_polygons}

For this discussion we fix a Riemannian metric $g$ on $S^2$. The following definitions are taken from~\cite{ce75}.

\begin{defn}
Let $X\subset S^2$.
\begin{itemize}
\item[i)] $X$ is strongly convex if for every pair of points $p,q$ in $X$ there is a unique minimal geodesic from $p$ to $q$, and this geodesic is contained in $X$.
\item[ii)] $X$ is convex if for every  $p$ in $\overline X$ there exists $r>0$ such that $B_r(p)\cap X$ is strongly convex.
\end{itemize}
\end{defn}

When $p\in S^2$ and $u,v \in T_pS^2$ are non-colinear vectors, consider the sets
\begin{align}
& \Delta(u,v)=\{ su+tv \mid s,t\geq 0 \} \label{sector} \\
& \Delta_r(u,v) = \{w\in \Delta(u,v) \mid |w|<r\}. \label{sector_r}
\end{align}
When $u\in T_pS^2\setminus\{0\}$ consider also 
\begin{align}
& H(u)=\{ v\in T_p S^2 \mid g(v,u)\geq 0 \} \label{half_space} \\
& H_r(u) = \{w\in H(u) \mid |w|<r\}. \label{half_space_r}
\end{align}

A corner of a unit speed broken geodesic $\gamma:\R/L\Z\to S^2$ is a point $\gamma(t)$ such that $\gamma_+'(t)\not\in\R^+\gamma_-'(t)$, where $\gamma'_\pm$ denote one-sided derivatives.

\begin{defn}\label{def_polygons}
$D\subset S^2$ is said to be a geodesic polygon if it is the closure of an open disk bounded by a simple closed unit speed broken geodesic $\gamma:\R/L\Z\to S^2$. We call $D$ convex if for every corner $p=\gamma(t)$ of $\gamma$ we find $0<r<{\rm inj}_p$ small enough such that $D\cap B_r(p) = \exp_p(\Delta_r(-\gamma_-'(t),\gamma_+'(t)))$. The corners of $\gamma$ are called vertices of $D$, and a side of $D$ is a smooth geodesic arc contained in $\partial D$ connecting two adjacent vertices.
\end{defn}

Jordan's theorem ensures that every simple closed unit speed broken geodesic is the boundary of exactly two geodesic polygons. At each boundary point which is not a vertex the inner normals to the two polygons are well-defined and opposite to each other. 

It is well-known that $B_r(p)$ is strongly convex when $r$ is small enough. By the following lemma the same property holds for $\exp_p(\Delta_r(u,v))$ and $\exp_p(H_r(u))$.

\begin{lem}\label{lemma_sector_str_convex}
Choose $p$ in $S^2$ and let $0<r<{\rm inj}(g)$. If $B_r(p)$ is strongly convex then $\exp_p(\Delta_r(u,v))$ and $\exp_p(H_r(u))$ are strongly convex for all pairs $u,v\in T_pS^2$ of non-colinear vectors.
\end{lem}

\begin{proof}
There is no loss of generality to assume that $u,v$ are unit vectors. We argue indirectly. Assume that $y,z\in\exp_p(\Delta_r(u,v))$ are points for which the minimal geodesic $\gamma$ from $y$ to $z$ (with unit speed) is not contained in $\exp_p(\Delta_r(u,v))$. Let $\gamma_u$ and $\gamma_v$ be the geodesic segments $\exp_p(\tau u)$, $\exp_p(\tau v)$ respectively, $\tau\in(-r,r)$. Note that $\gamma$ is contained in $B_r(p)$ and, consequently, $\gamma$ must intersect one of the geodesic segments $\gamma_u$ or $\gamma_v$ in two points $a\neq b$. Thus we have found two geodesic segments from $a$ to $b$ which are length minimisers in $S^2$ (one is contained in $\gamma$ and the other is contained in $\gamma_u$ or $\gamma_v$). This contradicts the fact that $B_r(p)$ is strongly convex. The argument to prove strong convexity of $\exp_p(H_r(u))$ is analogous.
\end{proof}

As an immediate consequence we have the following:

\begin{cor}
A convex geodesic polygon $D\subset S^2$ is convex.
\end{cor}

Let $d(p,q)$ denote the $g$-distance between points $p,q\in S^2$.

\begin{lem}\label{lemma_local_conv}
Let $D$ be a convex geodesic polygon. Then there exists  a positive number $\epsilon_1<{\rm inj}(g)$ such that if $p,q$ are in $D$ and satisfy $d(p,q) \leq \epsilon_1$, then the (unique) minimal geodesic from $p$ to $q$ lies in $D$.
\end{lem}

\begin{proof}
If not we find $p_n,q_n\in D$ such that $d(p_n,q_n)\to0$ and the minimal geodesic $\gamma_n$ in $S^2$ from $p_n$ to $q_n$ intersects $S^2\setminus D$. Thus, up to selection of a subequence, we may assume that $p_n,q_n \to x\in \partial D$. If $x$ is not a corner of $\partial D$ then we consider the unit vector $n\in T_xS^2$ pointing inside $D$ normal to the boundary and note that, for some $r>0$ small, $D\cap B_r(x)=\exp_x(H_r(n))$ is strongly convex. Here we used Lemma~\ref{lemma_sector_str_convex}. This is in contradiction to the fact that $p_n,q_n\in D\cap B_r(x)$ when $n$ is large. Similarly, if $x$ is a corner of $\partial D$ then, in view of the same lemma, we find unit vectors $u,v\in T_xS^2$ and $r$ very small such that $D\cap B_r(x)=\exp_x(\Delta_r(u,v))$ is strongly convex. This again provides a contradiction.
\end{proof}

The next lemma shows that a convex geodesic polygon is `convex in the large'.

\begin{lem}\label{lemma_polygon_large}
Let $D$ be a convex geodesic polygon. Then for every $p$ and $q$ in $D$ there is a smooth geodesic arc $\gamma$ from $p$ to $q$ satisfying
\begin{itemize}
\item[i)] $\gamma\subset D$.
\item[ii)] $\gamma$ minimises length among all piecewise smooth curves inside $D$ from $p$ to $q$.
\end{itemize}
\end{lem}

\begin{proof}
The argument follows a standard scheme. Consider a partition $P$ of $[0,1]$ given by $t_0=0<t_1<\dots<t_{N-1}<t_N=1$, with norm 
\[
\|P\|=\max_i\{t_{i+1}-t_{i}\}.
\] 
Let $\Lambda_P$ be the set of continuous curves $\alpha:[0,1]\to S^2$ such that each $\alpha|_{[t_i,t_{i+1}]}$ is smooth, $\alpha(0)=p$, $\alpha(1)=q$. On $\Lambda_P$ we have the usual length and energy functionals
\begin{equation}
\begin{array}{ccc}
L[\alpha]=\int_0^1 |\alpha'(t)|dt, & & E[\alpha]=\frac{1}{2} \int_0^1 |\alpha'(t)|^2dt.
\end{array}
\end{equation}
Set 
\begin{eqnarray*}
& B_P=\{\alpha\in\Lambda_P \mid \alpha|_{[t_i,t_{i+1}]} \ \text{ is a geodesic} \ \forall i\}, & \\ & \Lambda_P(D) = \{ \alpha\in\Lambda_P \mid \alpha([0,1])\subset D\}, \quad B_P(D)=B_P\cap\Lambda_P(D). &
\end{eqnarray*}
As usual, we use superscritps $\leq a$ to indicate sets of paths satisfying $E\leq a$.

If $\alpha$ is in $\Lambda_P^{\leq a}$ and $\sqrt{\|P\|} \leq \epsilon_1/\sqrt{2a}$, then $d(\alpha(t_i),\alpha(t_{i+1})) \leq \epsilon_1 \ \forall i$, where $\epsilon_1>0$ is the number given by Lemma~\ref{lemma_local_conv}. Thus, for every $\alpha \in \Lambda^{\leq a}_P(D)$ we find $\gamma \in B_P(D)$ such that each $\gamma|_{[t_i,t_{i+1}]}$ is a constant-speed reparametrization of the unique minimal geodesic arc from $\alpha(t_i)$ to $\alpha(t_{i+1})$. Here we have used Lemma~\ref{lemma_local_conv} to conclude that $\gamma([0,1]) \subset D$.  Clearly $L[\gamma]\leq L[\alpha]$, so minimizing $L$ on $\Lambda_P^{\leq a}(D)$ amounts to minimizing $L$ on $B_P^{\leq a}(D)$. Now pick $a>0$ and a partition $P$ such that $\Lambda_P^{\leq a}(D)\neq \emptyset$ and $\sqrt{\|P\|} \leq \epsilon_1/\sqrt{2a}$. By the above argument, $B_P^{\leq a}(D)\neq\emptyset$ and, as usual, the map $\gamma \mapsto (\gamma(t_1),\dots,\gamma(t_{N-1}))$ is a bijection between $B_P^{\leq a}(D)$ and a certain closed subset of $D^{N-1}$. The topology which $B_P^{\leq a}(D)$ inherits from this identification makes $L$ continuous. Thus, by compactness, we find $\gamma_*\in B^{\leq a}_P(D)$ which is an absolute minimiser of $L$ over $\Lambda_P^{\leq a}(D)$.

We claim that $\gamma_*$ is smooth, i.e., it has no corners. In fact, arguing indirectly, suppose it has a corner, which either lies on ${\rm int}(D)$ or on $\partial D$. In both cases we can use the auxiliary claim below to find a variation of $\gamma_*$ through paths in $B^{\leq a}_P(D)$ that decreases length; the convexity of $D$ is strongly used. This is a contradiction, and the smoothness of $\gamma_*$ is established. \\

\noindent {\bf Auxiliary Claim.} Consider $a<x<b$ and a broken geodesic $\beta:[a,b]\to S^2$, which is smooth and non-constant on $[a,x]$ and on $[x,b]$, satisfying $\beta'_+(x) \not\in \R^+\beta'_-(x)$. Let $\alpha:(-\epsilon,\epsilon)\times[a,b]\to S^2$ be a piecewise smooth variation with fixed endpoints of $\beta$ ($\alpha(0,\cdot)=\beta$) by broken geodesics such that $\alpha$ is smooth on $(-\epsilon,\epsilon)\times[a,x]$ and on $(-\epsilon,\epsilon)\times[x,b]$.  If $ D_1\alpha(0,x)$  is a non-zero vector in $\Delta(-\beta'_-(x),\beta'_+(x))$, then $\frac{d}{ds}|_{s=0}L[\alpha(s,\cdot)]<0$. In fact, the first variation formula gives us
\begin{equation*}
 \frac{d}{ds} \int_a^b|D_2\alpha(s,t)|\, dt  \Big|_{s=0} = g_{\beta(x)}\left( D_1\alpha(0,x), \frac{\beta_-'(x)}{\|\beta_-'(x)\|} - \frac{\beta_+'(x)}{\|\beta_+'(x)\|} \right) < 0
\end{equation*}
as desired.  \qed \\

It remains to be shown that $\gamma_*$ is an absolute length minimiser among all piecewise smooth curves in $D$ joining $p$ to $q$. Let $\alpha$ be such a curve, which must belong to $\Lambda_Q^{\leq b}(D)$ for some positive number $b$ and some partition $Q$. Up to increasing $b$ and refining $Q$, we may assume that $b\geq a$, $Q\supset P$, and $\sqrt{\|Q\|} \leq \epsilon_1/\sqrt{2b}$. By the previously explained arguments we can find a smooth geodesic $\tilde\gamma$ from $p$ to $q$ in $D$ which is a global minimiser of $L$ over $\Lambda^{\leq b}_Q(D)$. Since $\Lambda^{\leq a}_P(D)$ is contained in $\Lambda^{\leq b}_Q(D)$, we must have $L[\tilde\gamma]\leq L[\gamma_*]$. Noting that $\gamma_*,\tilde\gamma$ are smooth geodesics, we compute $E[\tilde\gamma]=\frac{1}{2}L[\tilde\gamma]^2\leq \frac{1}{2}L[\gamma_*]^2= E[\gamma_*]$ and conclude that $\tilde\gamma\in \Lambda_P^{\leq a}(D)$. Thus $L[\gamma_*]=L[\tilde\gamma]\leq L[\alpha]$ as desired.
\end{proof}

\begin{lem}\label{lemma_subdiv}
If $D$ is a convex geodesic polygon in $(S^2,g)$, $p$ and $q$ are distinct points of $\partial D$, and $d$ is the distance from $p$ to $q$ relative to $D$ then the following holds: a unit speed geodesic $\gamma:[0,d]\to D$ from $p$ to $q$ minimal relative to $D$ (which exists and is smooth in view of Lemma~\ref{lemma_polygon_large}) is injective, and satisfies either $\gamma((0,d)) \subset{\rm int}(D)$ or $\gamma([0,d])\subset\partial D$. In the former case $\gamma$ divides $D$ into two convex geodesic polygons $D',D''$ satisfying $D=D'\cup D''$, $\gamma=D'\cap D''$; moreover, a geodesic between two points of $D'$ ($D''$) which is minimal relative to $D$ is contained in $D'$ ($D''$). In the latter case there are no vertices of $D$ in $\gamma((0,d))$.
\end{lem}

\begin{proof}
If there exists $t$ in $(0,d)$ such that $\gamma(t)$ belongs to $\partial D$, then either $\gamma(t)$ is a vertex or not. But it can not be a vertex since in this case $\gamma'(t)$ would be colinear to one of the tangent vectors of $\partial D$ at $\gamma(t)$, allowing us to find $t'$ close to $t$ such that $\gamma(t')$ is not in $D$. Not being a vertex, $\gamma(t)$ is a point of tangency with $\partial D$. By uniqueness of solutions of ODEs, we must have $\gamma([0,d])\subset\partial D$, hence $D$ has no vertices in $\gamma((0,d))$. By minimality $\gamma$ has to be injective. If $\delta',\delta''$ are the two distinct arcs on $\partial D$ from $p$ to $q$ and $\gamma((0,d)) \cap \partial D=\emptyset$ then $\delta'\cup\gamma$ and $\delta''\cup\gamma$ bound disks $D',D''\subset D$ which are clearly geodesic convex polygons. Let $\alpha\subset D$ be a (smooth) geodesic arc connecting distinct points of $D'$ minimal relative to $D$. If $\alpha\not\subset D'$ then $\alpha$ intersects $\gamma((0,d))$ transversally at (at least) two distinct points $x\neq y$. By minimality, there are subarcs of $\alpha$ and of $\gamma$ from $x$ to $y$ with the same length. Thus, one can use these transverse intersections in a standard fashion to find a smaller curve in $D$ connecting the end points of $\alpha$, contradicting its minimality.
\end{proof}

\begin{lem}\label{lemma_rel_diam}
If the Gaussian curvature of $g$ is everywhere not smaller than $H>0$ then any two points $p,q\in D$ can be joined by a smooth geodesic arc $\gamma$ satisfying $\gamma\subset D$, $L[\gamma]\leq \pi/\sqrt H$.
\end{lem}
\begin{proof}
According to Lemma~\ref{lemma_polygon_large} we can find a smooth geodesic arc $\gamma:[0,1]\to D$ from $p$ to $q$ which is 
length minimizing among all piecewise smooth curves from $p$ to $q$ inside~$D$. If $L[\gamma]>\pi/\sqrt H$ then for every $\epsilon>0$ small enough we can find $t_\epsilon \in (\epsilon,1)$ such that $\gamma(t_\epsilon )$ is conjugated to $\gamma(\epsilon)$ along $\gamma|_{[\epsilon,t_\epsilon ]}$. Note that either $\gamma$ is contained in a single side of $D$ or $\gamma$ maps $(0,1)$ into ${\rm int}(D)$. In latter case we use a Jacobi field $J$ along $\gamma|_{[\epsilon,t_\epsilon ]}$ satisfying $J(\epsilon)=0$, $J(t_\epsilon )=0$ to construct an interior variation of $\gamma$ which decreases length, a contradiction. In the former note that, perhaps up to a change of sign, $J$ can be arranged so that it produces variations into $D$ which decrease length, again a contradiction.
\end{proof}

Before moving to Toponogov's theorem and its consequence, we take a moment to study convex geodesic polygons on the $2$-sphere equipped with its metric of constant curvature $H>0$. This space is realised as a spherical shell of radius $H^{-1/2}$ sitting inside the euclidean $3$-space, and will be denoted by $S_H$.

\begin{lem}\label{lemma_polygons_const_curv}
Let $D$ be a convex geodesic polygon in $S_H$.  
Then the following hold.
\begin{itemize}
\item[i)] $D$ coincides with the intersection of the hemispheres determined by its sides and the corresponding inward-pointing normal directions.
\item[ii)] The total perimeter of $\partial D$ is not larger than $2\pi/\sqrt H$.
\item[iii)] If $D$ has at least two sides then all sides of $D$ have length at most $\pi/\sqrt H$.
\end{itemize}
\end{lem}

\begin{proof}
Assertion iii) is obvious.  The argument to be given below to prove i) and ii) is by induction on the number $n$ of sides of $D$. The cases $n=1,2,3$ are obvious.

Now fix $n>3$ and assume that i), ii) and iii) hold for cases with $j<n$ sides. Let $p,q,r$ be three consecutive vertices of $D$, so that minimal geodesic arcs $\gamma_{pq},\gamma_{qr}$ from $p$ to $q$ and from $q$ to $r$, respectively, can be taken as two consecutive sides of $D$.   Here we used that sides have length at most $\pi/\sqrt H$.   Let $\gamma_1,\dots,\gamma_{n-2}$ be the other sides of $D$ and denote by $H_{pq},H_{qr},H_1,\dots,H_{n-2}$ the corresponding hemispheres determined by these sides and $D$. 

We argue indirectly to show that $D\subset H_{pq}\cap H_{qr}$. If $x\in D\setminus (H_{pq}\cap H_{qr})$, consider a smooth geodesic arc $\gamma$ from $x$ to $q$ inside $D$ which minimises length among piecewise smooth paths in $D$. $\gamma$ exists by Lemma~\ref{lemma_polygon_large} and, by the Lemma~\ref{lemma_rel_diam}, $L[\gamma]\leq\pi/\sqrt H$. Since $x$ is not antipodal to $q$ we have $L[\gamma]<\pi/\sqrt H$ which implies that $\gamma$ is the unique minimal geodesic from $x$ to $q$ in $S_H$. Combining $x\not\in H_{pq}\cap H_{qr}$ and Definition~\ref{def_polygons} one concludes that $\gamma$ is not contained in $D$, a contradiction.  Repeating this argument for all triples of consecutive vertices we find that 
\begin{equation}\label{inc_hemi_1}
D \subset H_{pq} \cap H_{qr} \cap H_1\cap \dots \cap H_{n-2}.
\end{equation}

Now let $\gamma_{pr}\subset D$ be the smooth geodesic arc from $p$ to $r$ which is minimal relatively to $D$. This arc exists by Lemma~\ref{lemma_polygon_large}. Moreover, $\gamma_{pr} \setminus \{p,r\} \subset {\rm int}(D)$ since otherwise, by the previous lemma, $\gamma_{pr}\subset \partial D$ contradicting the fact that $n>3$. Note that $\gamma_{pr}$ divides $D$ into $D=D'\cup T$, where $D'$ is a convex geodesic polygon with sides $\gamma_{pr},\gamma_1,\dots,\gamma_{n-2}$, and $T$ is the convex geodesic triangle bounded by $\gamma_{pq},\gamma_{qr},\gamma_{pr}$. Finally, let $H_{pr}$ be the hemisphere determined by $\gamma_{pr}$ and $D'$, and let $H'_{pq}$ be the closure of $S_H\setminus H_{pr}$. By the induction step $D'=H_{pr}\cap H_1\cap \dots\cap H_{n-2}$, and $T=H_{pq}\cap H_{qr} \cap H'_{pr}$. Thus 
\begin{equation}\label{inc_hemi_2}
\begin{aligned} 
&H_{pq}\cap H_{qr}\cap H_1\cap \dots\cap H_{n-2} \\
&= H_{pq}\cap H_{qr}\cap H_1\cap \dots\cap H_{n-2} \cap S_H \\
&= H_{pq}\cap H_{qr}\cap H_1\cap \dots\cap H_{n-2} \cap (H_{pr} \cup H'_{pr}) \\
&\subset (H_{pr}\cap H_1\cap \dots\cap H_{n-2}) \cup (H_{pq}\cap H_{qr}\cap H'_{pr}) \\
&= D' \cup T = D.
\end{aligned}
\end{equation}
Hence \eqref{inc_hemi_1} and~\eqref{inc_hemi_2} prove that i) holds for all convex geodesic polygons with at most $n$ sides. %To see that iii) holds for $D$ note that $\gamma_{pq}$ and $\gamma_{qr}$ have length at most $\pi/\sqrt H$ since each is contained in a distinct side of $H_{pq}\cap H_{qr}$. The other sides of $D$ are also sides of $D'$ which, by induction, have lengths at most $\pi/\sqrt H$, as desired.

To prove ii) we again assume $n>3$ and consider $a,b,c,d$ four consecutive vertices of $D$, the consecutive sides $\gamma_{ab},\gamma_{bc},\gamma_{cd}$ connecting them, and let $\gamma_1,\dots,\gamma_{n-3}$ be the other sides of $D$. Let $H_{bc}$ be the hemisphere containing $D$ whose equator contains $\gamma_{bc}$, and let $H_{bc}'$ be the closure of $S_H\setminus H_{bc}$. Continue $\gamma_{ab}$ along $b$ and $\gamma_{cd}$ along $c$ till they first meet at a point $e\in {\rm int}(H'_{bc})$. If $\gamma_{be},\gamma_{ec}$ are the minimal arcs connecting $b$ to $e$ and $e$ to $c$, respectively, and $T$ is the convex triangle with sides $\gamma_{be},\gamma_{ec},\gamma_{bc}$, then we claim that $F=D\cup T$ is a convex geodesic polygon with $n-1$ sides. To see this the reader will notice that the closed curve $\alpha=\gamma_{ab}\cup \gamma_{be} \cup \gamma_{ec}\cup \gamma_{cd} \cup \gamma_1\cup \dots\cup\gamma_{n-3}$ is simple since $T\subset H'_{bc}$ and $D\subset H_{bc}$ ($D$ satisfies i)), and $\alpha = \partial F$. By the induction step $\alpha$ has length $<2\pi/\sqrt H$ and, since $\gamma_{bc}$ is minimal, the length of $\partial D$ is smaller than that of~$\alpha$.
\end{proof}

\subsection{The Relative Toponogov's Theorem}\label{sec_rel_top_statement}

Toponogov's triangle comparison theorem is one of the most important tools in global Riemannian geometry. In the case of convex surfaces, it had been previously proven by Aleksandrov in \cite{ale48}.
Here we need a relative version for triangles in convex geodesic polygons sitting inside positively curved two-spheres.

We fix a metric $g$ on $S^2$, a convex geodesic polygon $D\subset S^2$, and follow~\cite{ce75} closely. However, we need to work with distances relative to $D$. For instance given points of $D$, the distance between them relative to $D$ is defined to be the infimum of lengths of piecewise smooth paths in $D$ connecting these points. Lemma~\ref{lemma_local_conv} tells us that the relative distance is realised by a smooth geodesic arc contained in~$D$. We say that a (smooth) geodesic arc between two points of $D$ is minimal relative to $D$ if it realises the distance relative to $D$.

A geodesic triangle in $D$ is a triple of non-constant geodesic arcs $(c_1,c_2,c_3)$ parametrised by arc-length, $c_i:[0,l_i]\to S^2$ ($l_i$ is the length of $c_i$), satisfying $c_i([0,l_i])\subset D$, $c_i(l_i)=c_{i+1}(0)$ and the triangle inequalities $l_i\leq l_{i+1}+l_{i+2}$ (indices modulo $3$). These arcs may or may not self-intersect and intersect each other. The angle $\alpha_i\in[0,\pi]$ is defined as the angle between $-c_{i+1}'(l_{i+1})$ and $c_{i+2}'(0)$ (indices modulo~$3$).

\begin{thm}[Relative Toponogov's Theorem]\label{thm_rel_top}
Let $g$ be a Riemannian metric on $S^2$ with Gaussian curvature pointwise bounded from below by a constant $H>0$, and let $D\subset S^2$ be a convex geodesic polygon. If $(c_1,c_2,c_3)$ is a geodesic triangle in $D$ such that $c_1,c_3$ are minimal relative to $D$ and $l_2\leq \pi/\sqrt H$, then for every $0<\epsilon<H$ there exists a so-called comparison triangle $(\bar c_1,\bar c_2,\bar c_3)$ in $S_{H-\epsilon}$ with angles $\bar\alpha_1,\bar\alpha_2,\bar\alpha_3$ such that $L[c_i]=L[\bar c_i]$ and $\bar\alpha_i\leq \alpha_i$, where $\alpha_i$ are the angles of $(c_1,c_2,c_3)$.
\end{thm}

In~\cite[page 297]{kli82} Klingenberg observes that the relative version of Toponogov's theorem holds, and that this observation is originally due to Alexandrov~\cite{ale48}. A proof of the above theorem would be too long to be included here, but the reader familiar with the arguments from~\cite{ce75} will notice two facts:
\begin{itemize}
\item The proof from~\cite{ce75} for the case of complete Riemannian manifolds essentially consists of breaking the given triangle into many `thin triangles' (these are given precise definitions in~\cite[chapter 2]{ce75}), and the analysis of these thin triangles is done by estimating lengths of arcs which are $C^0$-close to them. Hence all estimates of the perimeters of these thin triangles are obtained relative to an arbitrarily small neighborhood of the given convex geodesic polygon.
\item Distances relative to the convex geodesic polygon are only at most a little larger than distances relative to a very small neighborhood of the convex geodesic polygon. This is easy to prove since we work in two dimensions.
\end{itemize}
Putting these remarks together the relative version of Toponogov's theorem can be proved using the arguments from~\cite{ce75}.

\begin{rem}\label{rmk_useful_triangles}
A geodesic triangle in $S_{H-\epsilon}$ with sides of length at most $\pi/\sqrt H$, either is contained in a great circle, or its sides bound a convex geodesic polygon.
\end{rem}

\subsection{The perimeter of a convex geodesic polygon}

\begin{thm}\label{thm_est_conv_polygon}
Let $(S^2,g)$ be a Riemannian two-sphere such that the Gaussian curvature is everywhere bounded from below by $H>0$. If $D$ is a convex geodesic polygon in $(S^2,g)$ then the perimeter of $\partial D$ is at most $2\pi/\sqrt H$. The same estimate holds for the perimeter of a two-gon consisting of two non-intersecting simple closed geodesic loops based at a common point.
\end{thm}

This is proved in~\cite[page 297]{kli82} for the case $\partial D$ is a closed geodesic (no vertices). We reproduce the argument  here, observing that it also works for the general convex geodesic polygon.

\begin{proof}[Proof of Theorem~\ref{thm_est_conv_polygon}]

Let $d>0$ be the perimeter of $\partial D$. We can parametrise $\partial D$ as the image of a closed simple curve $c:\R/d\Z\to S^2$ which is a broken unit speed geodesic. For each $n\geq 1$ and $k\geq 0$ we denote by $\gamma_{k,2^n}$ a (smooth) geodesic arc from $c(kd2^{-n})$ to $c((k+1)d2^{-n})$ in $D$ which minimises length relative to $D$. We make these choices $2^n$-periodic in $k$, $\gamma_{k+2^n,2^n}=\gamma_{k,2^n}$, and also choose $\gamma_{0,2}=\gamma_{1,2}$. We can assume that $L[\gamma_{0,2}] < d/2$ since, otherwise, $d/2\leq L[\gamma_{0,2}] \leq \pi/\sqrt H$ (Lemma~\ref{lemma_rel_diam}) and the proof would be complete. In particular, $\gamma_{0,2}$ is not contained in $\partial D$, and Lemma~\ref{lemma_subdiv} implies that $\gamma_{0,2}$ touches $\partial D$ only at its endpoints $c(0),c(d/2)$.

Notice that if the distance from $c(kd2^{-n})$ to $c((k+1)d2^{-n})$ relative to $D$ is $d2^{-n}$, then Lemma~\ref{lemma_subdiv} implies that $c|_{[kd2^{-n},(k+1)d2^{-n}]}$ is a smooth geodesic arc. Therefore, we are allowed to make the following important choice:

\medskip

\noindent (C) If the distance from $c(kd2^{-n})$ to $c((k+1)d2^{-n})$ relative to $D$ is $d2^{-n}$, then we choose $\gamma_{k,2^n}=c|_{[kd2^{-n},(k+1)d2^{-n}]}$.

\medskip

The above choice forces $\gamma_{l,2^{n+m}}$ to be $c|_{[ld2^{-n-m},(l+1)d2^{-n-m}]}$ for all $k2^m\leq l<(k+1)2^m$, whenever $\gamma_{k,2^n}=c|_{[kd2^{-n}(k+1)d2^{-n}]}$.

For $n\geq 2$ set $D_n$ to be the subregion of $D$ bounded by the simple closed broken geodesic $\partial D_n = \cup \{\gamma_{k,2^n}\mid 0\leq k <2^n\}$. It follows readily from Lemma~\ref{lemma_subdiv} that this is a convex geodesic polygon. Moreover, sides of $D_n$ fall into two classes: either a side is not contained in $\partial D$ and coincides precisely with $\gamma_{k,2^n}$ for some $k$, or it lies in $\partial D$ is a union of adjacent $\gamma_{k,2^n}\cup\gamma_{k+1,2^n}\cup \dots\cup \gamma_{k+m,2^n} \subset\partial D$ for some $k$ and some $m$. By construction
\begin{itemize}
\item[i)] $D_n\subset D_{n+1}$ and $L[\partial D_n]\to d$ as $n\to\infty$.
\item[ii)] The vertices of $D_n$ form a subset of $\{c(kd2^{-n})\mid 0\leq k<2^n\}$.
\end{itemize}

Fix $0<\epsilon<H$. We would like to construct a sequence of convex geodesic polygons $E_n \subset E_{n+1}$ in $S_{H-\epsilon}$ such that $L[\partial E_n] = L[\partial D_n]$.

Consider geodesic triangles $T_{k,2^n}=(\gamma_{k,2^n},\gamma_{2k,2^{n+1}},\gamma_{2k+1,2^{n+1}})$ in the sense of \S\ref{sec_rel_top_statement}. The triangle inequalities hold, since all sides are minimal relative to $D$.

According to Theorem~\ref{thm_rel_top}, associated to $T_{0,2},T_{1,2}$ there are comparison triangles $\bar T_{0,2} = (\bar \gamma_{0,2},\bar\gamma_{0,4},\bar\gamma_{1,4})$, $\bar T_{1,2} = (\bar\gamma_{1,2},\bar\gamma_{2,4},\bar\gamma_{3,4})$ in $S_{H-\epsilon}$ with sides of same length as the corresponding sides in $T_{0,2},T_{1,2}$. The angles of $\bar T_{0,2},\bar T_{1,2}$ are not larger than the corresponding angles on $T_{0,2},T_{1,2}$. Up to reflection and a rigid motion, we can assume $\bar\gamma_{0,2}$ coincides with $\bar\gamma_{1,2}$ (along with vertices corresponding to endpoints of $\gamma_{0,2}=\gamma_{1,2}$) on a given great circle $e$, and $\bar T_{0,2},\bar T_{1,2}$ lie on opposing hemispheres determined by~$e$. Of course, $\bar T_{0,2}$ and/or $\bar T_{1,2}$ could lie on $e$, but this forces $L[\gamma_{0,2}]$ to be $d/2$, a case we already treated. Again the angle comparison can be used to deduce that $E_2 := \bar T_{0,2}\cup \bar T_{1,2}$ is a convex geodesic polygon in $S_{H-\epsilon}$ with the same perimeter as $D_2$ ($\partial E_2 = \cup_{k=0}^3 \bar\gamma_{k,4}$).

To construct $E_3$, note that each side of $D_2$ not contained in $\partial D$ is of the form $\gamma_{k,4}$ for some fixed $0\leq k<4$. Moreover, $\bar\gamma_{k,4}$ is a side of $E_2$ by construction and angle comparison. By Lemma~\ref{lemma_subdiv} $\gamma_{k,4}$ divides $D$ into two convex geodesic polygons, only one of which, denoted by $D_{k,4}$, contains $c([kd/4,(k+1)d/4])$ in its boundary. By the same lemma, $T_{k,4}$ is contained in $D_{k,4}$ (and determines a convex geodesic polygon). By the relative Toponogov theorem, there exists a comparison triangle $\bar T_{k,4}$ which we can assume is of the form $(\bar\gamma_{k,4},\bar\gamma_{2k,8},\bar \gamma_{2k+1,8})$, i.e. one of its sides matches precisely the side $\bar\gamma_{k,4}$ of $E_2$ together with corresponding vertices of $\bar\gamma_{k,4}$. Moreover, possibly after reflection, we can assume $E_2$ and $\bar T_{k,4}$ lie on the opposing hemispheres determined by the great circle containing $\bar\gamma_{k,4}$. This last step strongly uses Lemma~\ref{lemma_polygons_const_curv} and Remark~\ref{rmk_useful_triangles}. Again by the angle comparison, $E_2\cup\bar T_{k,4}$ is a convex geodesic polygon in $S_{H-\epsilon}$ with the same perimeter as the convex geodesic polygon $D_2\cup T_{k,4}$. Repeating this procedure for another side of $D_2$ not in $\partial D$, which is of the form $\gamma_{k',4}$ for some $k'\neq k$, with $E_2\cup\bar T_{k,4}$ in the place of $E_2$, we obtain a larger geodesic convex polygon $E_2\cup\bar T_{k,4}\cup \bar T_{k',4}$ in $S_{H-\epsilon}$ with the same perimeter as the geodesic convex polygon $D_2\cup T_{k,4}\cup T_{k',4}$. After exhausting all the sides of $D_2$ not in $\partial D$ we complete the construction of $E_3$.

The construction of $E_n$ from $D_{n-1},E_{n-1}$ follows the same algorithm, since sides of $D_{n-1}$ not in $\partial D$ must be of the form $\gamma_{k,2^{n-1}}$ for some $0\leq k<2^{n-1}$. In this case, there will be a corresponding side $\bar\gamma_{k,2^{n-1}}$ of $E_{n-1}$ with the same length as $\gamma_{k,2^{n-1}}$ along which we fit the comparison triangle $\bar T_{k,2^{n-1}}$ obtained by applying the relative Toponogov theorem to $T_{k,2^{n-1}}$. Doing this step by step at each side of $D_{n-1}$ not in $\partial D$ we obtain $E_n$.

By Lemma~\ref{lemma_polygons_const_curv} we know that 
\[
L[\partial D_n] = L[\partial E_n]\leq 2\pi/\sqrt{H-\epsilon}, \qquad \forall n.
\]
Together with (i) above, we deduce that $L[\partial D]\leq 2\pi/\sqrt{H-\epsilon}$. Letting $\epsilon\downarrow0$ we get the desired estimate.

To get the estimate for the two-gon as in the statement note that its perimeter can clearly be approximated by the perimeter of convex geodesic polygons.
\end{proof}

\section{Zoll geodesic flows on the two-sphere}
\label{app_conj_zoll}

Given a Riemannian metric $g$ on $S^2$, we denote by $T^1 S^2(g)$ the corresponding unit tangent bundle. The Hilbert 1-form on $TS^2$ is the pull-back of the standard Liouville form $p\, dq$ on $T^* S^2$ by the isomorphism $TS^2 \cong T^* S^2$ induced by the metric $g$ (see also the end of Section \ref{extsec} for an equivalent definition). This 1-form restricts to a contact form $\alpha_g$ on $T^1 S^2(g)$ whose  Reeb flow is the geodesic flow on $T^1 S^2(g)$. We recall that the geodesic equation induces also a Hamiltonian flow on $T^* S^2$, which is determined by the standard symplectic structure on $T^* S^2$ and by the Hamiltonian
\[
H_g(x,p) := \frac{1}{2} \, g^*_x(p,p), \qquad \forall (x,p)\in T^* S^2,
\]
where $g^*$ is the metric on the vector bundle $T^* S^2$ which is dual to $g$. By pushing this Hamiltonian flow forward to $TS^2$ by the isomorphism $T^* S^2 \cong T S^2$ which is induced by $g$ and by restriction to tangent vectors of norm one, we obtain precisely the geodesic flow on $T^1 S^2(g)$. The aim of this appendix is to present a full proof of the following result:

\begin{thm}
\label{conju-zoll}
Let $g$ be a metric on $S^2$ all of whose geodesics are closed and have length $2\pi$. Then 
\begin{equation}
\label{area}
\mathrm{area}(S^2,g) = \mathrm{area}(S^2,g_{\mathrm{round}}) = 4\pi,
\end{equation}
and there is a diffeomorphism 
\[
\varphi: T^1 S^2(g_{\mathrm{round}}) \rightarrow T^1 S^2(g)
\]
such that $\varphi^* \alpha_g = \alpha_{g_{\mathrm{round}}}$. In particular, $\varphi$ conjugates the geodesic flows of $g_{\mathrm{round}}$ and $g$. Furthermore, there is a symplectomorphism
\[
\psi: T^* S^2 \setminus \mathbb{O} \rightarrow T^* S^2 \setminus \mathbb{O}
\]
such that $\psi^* H_g = H_{g_{\mathrm{round}}}$. Here $\mathbb{O}$ denotes the zero section of $T^* S^2$. In particular, $\psi$ conjugates the Hamiltonian flows of $H_{g_{\mathrm{round}}}$ and $H_g$ away from the zero section.
\end{thm}

The statement about the area of $g$ is proved (for more general Zoll manifolds) by Weinstein in \cite{wei74}. The existence of a conjugacy is also proved by Weinstein (again for more general Zoll manifolds) in \cite{wei75}, but assuming the existence of a path of Zoll metrics connecting $g_{\mathrm{round}}$ to $g$. See also \cite{gui76}[Appendix B]. In the special case of $S^2$ one does not need this assumption. 

Before proving Theorem \ref{conju-zoll}, we study the contact manifold $(T^1 S^2(g_{\mathrm{round}}), \alpha_{g_{\mathrm{round}}})$. If we see $(S^2,g_{\mathrm{round}})$ as the unit sphere in $\R^3$, the unit tangent bundle $T^1 S^2(g_{\mathrm{round}})$ is naturally identified with the three-dimensional submanifold of $\R^6$
\[
\{(x,u)\in \R^3 \times \R^3 \mid |x|=|u|=1, \; x\cdot u = 0\} ,
\]
where $|\cdot|$ and $\cdot$ are the Euclidean norm and scalar product on $\R^3$. Using this identification, the contact form $\alpha_{g_{\mathrm{round}}}$ has the form
\[
\alpha_{g_{\mathrm{round}}}(x,u)[(v,w)] = u\cdot v, \qquad \forall (x,u)\in T^1 S^2(g_{\mathrm{round}}), \; (v,w)\in T_{(x,u)} T^1 S^2(g_{\mathrm{round}}).
\]
The above identification shows that $T^1 S^2(g_{\mathrm{round}})$ is diffeomorrphic to $SO(3)$ by the diffeomorphism
\[
T^1 S^2(g_{\mathrm{round}}) \rightarrow SO(3), \qquad (x,u) \mapsto [ x \;\; u \;\; x\times u ],
\]
where $\times$ is the vector product on $\R^3$ and $[a \; b \; c]$ denotes the matrix with columns $a,b,c$. The push-forward of $\alpha_{g_{\mathrm{round}}}$ by this diffeomorphism is the following contact form on $SO(3)$
\[
\alpha_0 (A) [H] := Ae_2 \cdot He_1, \qquad \forall A\in SO(3), \; H\in T_A SO(3),
\]
where $\{e_1,e_2,e_3\}$ is the standard basis of $\R^3$. Its differential is the two-form
\begin{equation}
\label{da0}
d\alpha_0(A)[H,K] = H e_2 \cdot K e_1 - K e_2 \cdot H e_1, \qquad \forall A\in SO(3), \; H,K\in T_A SO(3).
\end{equation}
On $SO(3)$ the geodesic flow of $g_{\mathrm{round}}$ takes the form
\[
\phi_t (A) = A \, R(t), \qquad \mbox{where} \quad R(t) := \left( \begin{array}{ccc} \cos t & - \sin t & 0 \\ \sin t & \cos t & 0 \\ 0 & 0 & 1 \end{array} \right).
\]
The flow $\phi_t$ defines a free $\T$-action on $SO(3)$, where $\T := \R/2\pi \Z$.
The quotient of $SO(3)$ by this $\T$-action is $S^2$, and the quotient projection is the map
\begin{equation}
\label{p0}
p_0 : SO(3) \rightarrow S^2, \qquad p_0(A) = A e_3.
\end{equation}
Denote by $\omega_0$ the standard area form of $S^2$, namely
\[
\omega_0(x)[u,v] := \det [ x \;\; u \;\; v ], \qquad \forall x\in S^2, \; u,v\in T_x S^2.
\]
We claim that 
\begin{equation}
\label{areas}
p_0^*\,  \omega_0 = - d\alpha_0.
\end{equation}
In order to prove this identity, notice that $d\alpha_0$ is invariant under the  action of $SO(3)$ by left multiplication: if $T\in SO(3)$ and $L_T$ is the map
\[
L_T : SO(3) \rightarrow SO(3) , \qquad L_T (A) = TA,
\]
then formula (\ref{da0}) shows that $L_T^* d\alpha_0 = d\alpha_0$. Moreover, from the identity $p_0 \circ L_T = T \circ p_0$ and from the fact that $\omega_0$ is $T$-invariant we deduce that also $p_0^* \omega$ is $L_T$-invariant:
\[
L_T^*(p_0^* \omega_0) = (p_0 \circ L_T)^* \omega_0 = (T \circ p_0)^* \omega_0 = p_0^*(T^* \omega_0) = p_0^* \omega_0.
\]
Therefore, it is enough to check the validity of the identity (\ref{areas}) at the identity matrix $I\in SO(3)$. Let $H,K$ be two elements of the tangent space of $SO(3)$ at $I$, that is, two skew-symmetric matrices
\[
H = \left( \begin{array}{ccc} 0 & h_1 & h_2 \\ - h_1 & 0 & h_3 \\ -h_2 & - h_3 &  0 \end{array} \right), \qquad K = \left( \begin{array}{ccc} 0 & k_1 & k_2 \\ - k_1 & 0 & k_3 \\ -k_2 & - k_3 &  0 \end{array} \right).
\]
By (\ref{da0}) we have
\[
d\alpha_0(I)[H,K] = h_3 k_2 - h_2 k_3.
\]
On the other hand, from the form (\ref{p0}) of the projection $p_0$ we find
\[
p_0^* \omega (I) [H,K] = \omega(e_3)[H e_3,Ke_3] = \det \left( \begin{array}{ccc} 0 & h_2 & k_2 \\ 0 & h_3 & k_3 \\ 1 & 0 & 0 \end{array} \right) = h_2 k_3 - h_3 k_2.
\]
The above two identities conclude the proof of (\ref{areas}).

The map $p_0: SO(3) \rightarrow S^2$ defines a principal $\T$-bundle. The contact form $\alpha_0$ is a connection 1-form on this principal bundle. By (\ref{areas}) the curvature 2-form $d\alpha_0$ coincides with $-p_0^* \omega_0$, and hence the Euler class of $p_0$ is  $[\omega_0/2\pi]$. In particular, the Euler number of $p_0$ is
\[
\langle [\omega_0/2\pi], [S^2] \rangle = \frac{1}{2\pi} \int_{S^2} \omega_0 =  2.
\]
We will deduce Theorem \ref{conju-zoll} by the following general result:

\begin{thm}\label{thm_classification_Zoll_ctct_forms}
Let $\alpha$ be a contact form on $SO(3)$ such that all the orbits of the corresponding Reeb flow are periodic and have minimal period $2\pi$. Then 
\begin{equation}
\label{volume}
\mathrm{vol}(SO(3),\alpha \wedge d\alpha) = \mathrm{vol}(SO(3),\alpha_0 \wedge d\alpha_0) = 8  \pi^2,
\end{equation}
and there exists a diffeomorphism $\varphi: SO(3) \rightarrow SO(3)$ such that $\varphi^* \alpha = \alpha_0$.
\end{thm}

\begin{proof}[Proof of Theorem~\ref{thm_classification_Zoll_ctct_forms}]
First notice that the thesis is true for the contact form $\alpha=-\alpha_0$: indeed, (\ref{volume}) is trivial in this case, and the diffeomorphism
\[
SO(3) \rightarrow SO(3), \qquad A \mapsto A D \quad \mbox{where} \quad D:= \left( \begin{array}{ccc} 1 & 0 & 0 \\ 0 & -1 & 0 \\ 0 & 0 & -1 \end{array} \right),
\]
satisfies $\varphi^*(-\alpha_0) = \alpha_0$.

Now consider an arbitrary contact form $\alpha$ on $SO(3)$ satisfying the periodicity assumption.
Up to the application of an orientation reversing diffeomorphism of $SO(3)$, we may assume that $\alpha$ and $\alpha_0$ induce the same orientation, meaning that the volume forms $\alpha \wedge d\alpha$ and $\alpha_0 \wedge d\alpha_0$ differ by the multiplication by a positive function. 

The Reeb flow of $\alpha$ induces a smooth free $\mathbb{T}$-action on $SO(3)$. The quotient $B$ of $SO(3)$ by this action is a smooth closed surface. Denote by
\[
p: SO(3) \rightarrow B
\]
the quotient projection. By (the easy part of) a theorem of Boothby and Wang (\cite{bw58}, see also \cite[Theorem 7.2.5]{gei08}), $p$ is a principal $\T$-bundle, $\alpha$ is a connection 1-form on it, whose curvature form $\omega$ is an area form on $B$ satisfying
\[
p^* \omega = d\alpha.
\]
Moreover, the cohomology class $-[\omega/2\pi]$ is integral and coincides with the Euler class $e$ of the $\T$-bundle.

In particular, $B$ is orientable and from the exact homotopy sequence of fibrations
\[
\cdots \rightarrow \pi_2(SO(3)) = 0 \rightarrow \pi_2(B) \rightarrow \pi_1(\T)=\Z \rightarrow \pi_1(SO(3)) = \Z_2 \rightarrow \cdots
\]
we deduce that $\pi_2(B)=\Z$. Therefore, $B$ is the two-sphere $S^2$. From the Gysin sequence
\[
\begin{split}
\cdots \rightarrow H^1(SO(3);\Z) = 0 &\stackrel{p_*}{\longrightarrow} H^0(S^2;\Z) = \Z \stackrel{\cup e}{\longrightarrow} H^2(S^2;\Z) = \Z \rightarrow \\ &\stackrel{p^*}{\longrightarrow} H^2(SO(3);\Z) = \Z_2  \stackrel{p_*}{\longrightarrow} H^1(S^2;\Z) = 0 \rightarrow \cdots 
\end{split}
\]
we deduce that the cup product with the Euler class is the multiplication by $\pm 2$, i.e.\ the Euler number of the $\T$-bundle $p$ is $\pm 2$. Then, choosing any orientation on $SO(3)$, we can compute the total volume of $\alpha\wedge d\alpha$ by fiberwise integration
\[
\begin{split}
\mathrm{vol}(SO(3),\alpha\wedge d\alpha) &= \left| \int_{SO(3)} \alpha \wedge d\alpha \right| = \left| \int_{SO(3)} \alpha \wedge p^* \omega \right| = \left| \int_{S^2} p_*(\alpha) \omega \right| \\ &= 2\pi \left| \int_{S^2} \omega \right| = 4\pi^2   \left| \int_{S^2} \frac{\omega}{2\pi}  \right| = 8\pi^2,
\end{split}
\]
and (\ref{volume}) follows.
If we change $\alpha$ by $-\alpha$, then $\omega$ becomes $-\omega$ and hence the Euler number of $p$ changes sign. Since $\alpha$ and $-\alpha$ induce the same orientation on $SO(3)$, we may assume that the Euler number is $2$: if in this case we do have a diffeomorphism $\varphi$ such that $\varphi^* \alpha = \alpha_0$, then the same diffeomorphism pulls $-\alpha$ back to $-\alpha_0$, and we have already checked that $-\alpha_0$ can be pulled back to $\alpha_0$.

Therefore, in the sequel we assume that $\alpha$ and $\alpha_0$ induce the same orientation on $SO(3)$ and that the Euler number of $p$ is 2. Since the Euler number determines principal $\T$-bundles over $S^2$ (see \cite[Theorem 8]{kob56}) and since we have checked above that the Euler number of $p_0$ is 2, there is an isomorphism of principal $\T$-bundles
\[
\xymatrix{ SO(3) \ar^{\psi}[rr] \ar_{p_0}[dr] & & SO(3) \ar^{p}[dl] \\ & S^2 & }
\] 
In particular, $\psi$ is orientation preserving and intertwines generators of the $\T$-actions, that is, the Reeb vector fields of $\alpha_0$ and $\alpha$. Denote by $R$ the Reeb vector field of $\alpha_0$. Then $R$ is also the Reeb vector field of the contact form
\[
\alpha_1 := \psi^* \alpha_0.
\]
We claim that
\[
\alpha_t := t \alpha_1 + (1-t) \alpha_0
\]
is a contact form for every $t\in [0,1]$. Since $\psi$ is orientation preserving, we have
\[
\alpha_1 \wedge d\alpha_1 = \psi^*(\alpha_0 \wedge d\alpha_0) = f \, \alpha_0 \wedge d\alpha_0
\]
for some positive smooth function $f$. Fix some point $A$ in $SO(3)$ and let $H,K$ be a basis of $\ker \alpha_0(A)$ such that
\[
d\alpha_0 [H,K] = 1,
\]
where we are omitting to write the point $A$. Then $R=R(A),H,K$ is a basis of the tangent space of $SO(3)$ at $A$, and we have
\begin{equation}
\label{pez1}
\alpha_0 \wedge d\alpha_0 [R,H,K] = 1.
\end{equation}
Since $R$ is the Reeb vector field of $\alpha_1$, we also have
\begin{equation}
\label{pez2}
d\alpha_1 [H,K] = \alpha_1 \wedge d\alpha_1 [R,H,K] = f \, \alpha_0 \wedge d\alpha_0[ R,H,K] = f.
\end{equation}
Therefore
\begin{equation}
\label{pez3}
\alpha_0 \wedge d\alpha_1 [R,H,K] = d\alpha_1 [H,K] = f, \qquad
\alpha_1 \wedge d\alpha_0 [R,H,K] = d\alpha_0 [H,K] = 1.
\end{equation}
By (\ref{pez1}), (\ref{pez2}) and (\ref{pez3}) we obtain
\[
\begin{split}
\alpha_t \wedge d\alpha_t [R,H,K] &=  \bigl( t^2 \alpha_1 \wedge d\alpha_1 + (1-t)^2 \alpha_0 \wedge d\alpha_0 + t(1-t) \alpha_1 \wedge d\alpha_0 \\ & \quad + t(1-t) \alpha_0 \wedge d\alpha_1 \bigr) [R,H,K] \\ &= t^2 f + (1-t)^2 + t(1-t) + t(1-t) f = tf + 1 - t.
\end{split}
\]
Since the above quantity is positive for every $t\in [0,1]$, $\alpha_t$ is a contact form for $t$ in this range, as claimed.

Now we proceed using Moser's argument. Since $d\alpha_t$ is non-degenerate on $\ker \alpha_t$, we can find a unique (and hence smooth) vector field $Y_t$ taking values in $\ker \alpha_t$ and such that
\[
\imath_{Y_t} d\alpha_t|_{\ker \alpha_t} = (\alpha_0 - \alpha_1)|_{\ker \alpha_t}.
\]
Since both $\imath_{Y_t} d\alpha_t$ and $\alpha_0 - \alpha_1$ vanish on $R$, we can remove the restriction to $\ker \alpha_t$ from the above identity:
\begin{equation}
\label{YY}
\imath_{Y_t} d\alpha_t = \alpha_0 - \alpha_1.
\end{equation}
Let $\phi_t: SO(3) \rightarrow SO(3)$, $t\in [0,1]$, be the one-parameter family of diffeomorphisms which solves the equation
\[
\phi_0 = \mathrm{id}, \qquad \frac{d}{dt} \phi_t = Y_t (\phi_t).
\]
By Cartan's identity we get
\[
\frac{d}{dt} \phi_t^*\alpha_t = \phi_t^* \left( L_{Y_t}\alpha_t + \alpha_1-\alpha_0 \right) = \phi_t^* \left( \imath_{Y_t}d\alpha_t + d\imath_{Y_t}\alpha_t + \alpha_1 -\alpha_0 \right) = 0,
\]
where we have used (\ref{YY}) and the fact  that $\imath_{Y_t} \alpha_t = \alpha_t[Y_t]=0$, since $Y_t$ is a section of $\ker \alpha_t$. Since $\phi_0^* \alpha_0 = \alpha_0$, we deduce that $\phi_t^* \alpha_t = \alpha_0$ for every $t\in [0,1]$.
In particular, 
\[
\phi_1^* \psi^* \alpha = \phi_1^* \alpha_1 = \alpha_0,
\] 
and $\varphi := \psi\circ\phi_1$ is the required diffeomorphism.
\end{proof}

\begin{proof}[Proof of Theorem \ref{conju-zoll}]
Using an arbitrary diffeomorphism between $T^1 S^2(g)$ and $SO(3)$ we identify also $\alpha_g$ with a contact form $\alpha$ on $SO(3)$, which satisfies the assumptions of Theorem \ref{thm_classification_Zoll_ctct_forms}. By  Proposition \ref{prop_volume_area}  and (\ref{volume}) we have
\[
2\pi \, \mathrm{area}(S^2,g) = \mathrm{vol}(T^1 S^2(g), \alpha_g) = \mathrm{vol}(SO(3), \alpha) = 8\pi^2,
\]
which proves (\ref{area}). The existence of a diffeomorphism 
\[
\varphi: T^1 S^2(g_{\mathrm{round}}) \rightarrow T^1 S^2(g)
\]
such that $\varphi^* \alpha_g = \alpha_{g_{\mathrm{round}}}$ is an immediate consequence of Theorem \ref{thm_classification_Zoll_ctct_forms}. Since it intertwines the Reeb vector fields of $\alpha_{g_{\mathrm{round}}}$ and $\alpha_g$, this diffeomorphism conjugates the two geodesic flows. Let $\tilde{\varphi}$ be the induced diffeomorphism between the unit cotangent bundles of $S^2$ which are defined by the dual metrics $g^*_{\mathrm{round}}$ and $g^*$. The diffeomorphism $\tilde{\varphi}$ intertwines the two restrictions of the standard Liouville form of $T^* S^2$. The one-homogeneous extension
\[
\psi: T^* S^2 \setminus \mathbb{O} \rightarrow T^* S^2 \setminus \mathbb{O}, \qquad \psi(ru) = r \tilde{\varphi}(u) \quad \mbox{for } u\in T^* S^2, \; g^*_{\mathrm{round}}(u,u)=1,\; r>0,
\]
is a symplectomorphism and satisfies $\psi^* H_g = H_{g_{\mathrm{round}}}$.
\end{proof}

%\bibliographystyle{amsalpha}
%\bibliography{../../biblio/nonlinear}

\providecommand{\bysame}{\leavevmode\hbox to3em{\hrulefill}\thinspace}
\providecommand{\MR}{\relax\ifhmode\unskip\space\fi MR }
\providecommand{\MRhref}[2]{%
  \href{http://www.ams.org/mathscinet-getitem?mr=#1}{#2}
}
\providecommand{\href}[2]{#2}

\end{document}